\documentclass[11pt]{amsart}
\usepackage{latexsym,amssymb,amsfonts,amsmath,graphicx,fullpage,url,color,tikz}
\usepackage[utf8]{inputenc}
\usepackage[vcentermath,enableskew]{youngtab}
\usepackage[all]{xy}
\usepackage{verbatim}
\usepackage{ytableau}
\usepackage{makecell}
\usepackage[export]{adjustbox}
\usepackage{multicol}
\usepackage{textcomp}

\newtheorem{theorem}{Theorem}[section]
\newtheorem{proposition}[theorem]{Proposition}
\newtheorem{lemma}[theorem]{Lemma}
\newtheorem{corollary}[theorem]{Corollary}
\newtheorem*{BijThm}{Theorem~\ref{thm:bijections}}

\newtheorem*{EquivBij}{Theorem~\ref{thm:equivbij}}
\theoremstyle{definition}
\newtheorem{definition}[theorem]{Definition}
\newtheorem{openproblem}[theorem]{Open Problem}

\theoremstyle{remark}
\newtheorem{remark}[theorem]{Remark}
\numberwithin{equation}{section}
\usepackage{multicol}

\newcommand{\pasm}{\mathrm{PASM}}

\newcommand{\blue}[1]{\textcolor{blue}{#1}}

\newcommand*{\defeq}{\mathrel{\vcenter{\baselineskip0.5ex \lineskiplimit0pt
                     \hbox{\scriptsize.}\hbox{\scriptsize.}}}%
                     =}
\newcommand{\bbz}{\mathbb{Z}}
\newcommand{\bbr}{\mathbb{R}}

\author{Dylan Heuer}
\email{heuerd@msoe.edu}

\address{Milwaukee School of Engineering}
\title{Partial Alternating Sign Matrix Bijections and Dynamics}
\keywords{alternating sign matrix; bijection, sign matrix, partial alternating sign matrix}
\subjclass[2010]{05A05, 52B05}

\begin{document}
\begin{abstract}
We investigate analogues of alternating sign matrices, called partial alternating sign matrices. We prove bijections between these matrices and several other combinatorial objects. We use an analogue of Wieland's gyration on fully-packed loops, which we relate to the study of toggles and order ideals. Finally, we show that rowmotion on order ideals of a certain poset and gyration on partial fully-packed loop configurations are in equivariant bijection.
\end{abstract}

\maketitle

\section{Introduction}
\label{sec:intro}
Alternating sign matrices have a very interesting history, enumeration, and connection to other areas of mathematics and science \cite{Bressoud}. There exist beautiful bijections between alternating sign matrices and many other combinatorial objects, including \emph{monotone triangles}, \emph{height-function matrices}, \emph{fully-packed loop configurations}, \emph{square ice configurations} and \emph{order ideals} of a particular poset. The bijection with square ice configurations is especially useful, as it has revealed connections to physics which played a key role in one of the proofs of the enumeration of alternating sign matrices \cite{Kuperberg}.

In this paper, we study an analogue of alternating sign matrices, which we call \emph{partial alternating sign matrices}. We allow these matrices to be rectangular, and in doing so relax the condition that each row and column must sum to $1$. These matrices appear as the individual components of \emph{chained alternating sign matrices} (see \cite{Chained}), which was the original motivation for studying them. We hope this study of partial alternating sign matrices will prove useful for future research in the chained setting. In Section~\ref{sec:prelim} we provide a summary of known definitions and results. These may be useful to help understand the main results, found in Sections \ref{sec:pasm_bij} and \ref{sec:pasm_dyn}, where we study rectangular analogues of alternating sign matrices. In Section~\ref{sec:pasm_bij}, we explore several bijections, which in most cases are analogues of those studied in the usual alternating sign matrix setting (see \cite{Propp, Striker1}). These bijections culminate in our main result of the section, which is the following theorem. See Figure~\ref{allbij} for an example of each of the objects in the statement of the theorem.

\begin{BijThm}
There are explicit bijections between $m \times n$ partial alternating sign matrices, $(m,n)$-partial monotone triangles, $(m,n)$-partial height-function matrices, $(m,n)$-partial fully-packed loop configurations, $(m,n)$-rectangular ice configurations, the set of order ideals of $\textbf{P}_{m,n}$, and $(m,n)$-nests of osculating paths.
\end{BijThm}

\begin{figure}[hbtp]
\centering
\[
\left(\begin{array}{rrrr}
1  & 0 & 0  & 0\\
0  & 0 & 1  & 0\\
-1 & 1 & 0  & 0\\
1  & 0 & -1 & 1
\end{array}\right)
\hspace{.2in}
\begin{array}{ccccccc}
  &   &   & 1 &   &   &  \\
  &   & 1 &   & 3 &   &  \\
  & 0 &   & 2 &   & 3 &  \\
0 &   & 1 &   & 2 &   & 4
\end{array}
\hspace{.2in}
\begin{pmatrix}
0 & 1 & 2 & 3 & 4\\
1 & 2 & 3 & 4 & 3\\
2 & 3 & 2 & 3 & 2\\
3 & 4 & 3 & 2 & 3\\
4 & 3 & 4 & 3 & 2
\end{pmatrix}
\]
\includegraphics[width=2in]{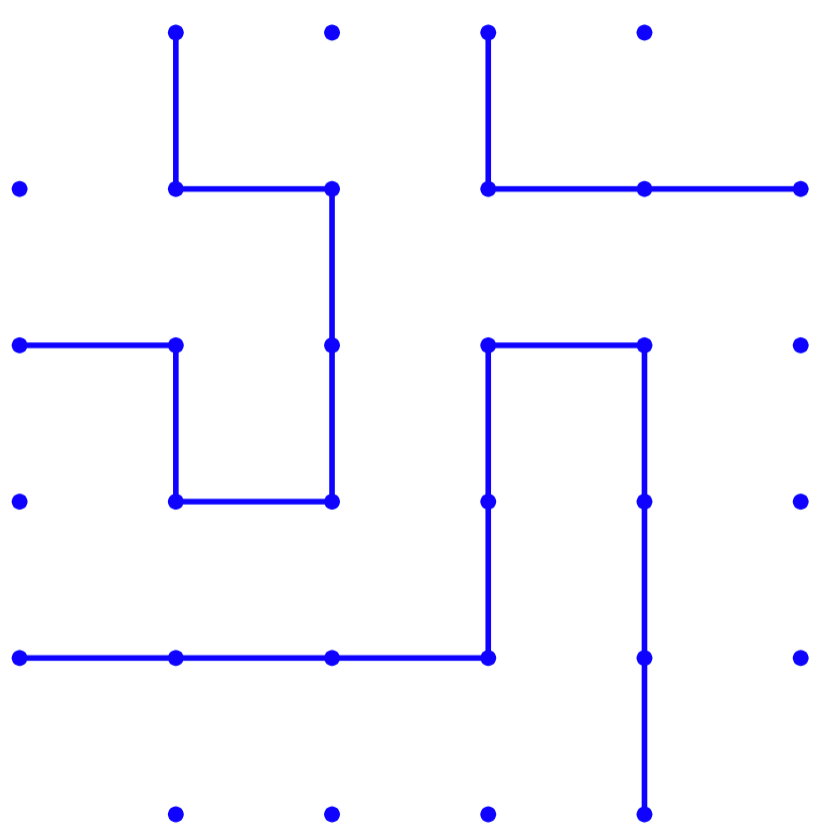} \hspace{0.25in} \includegraphics[width=2in]{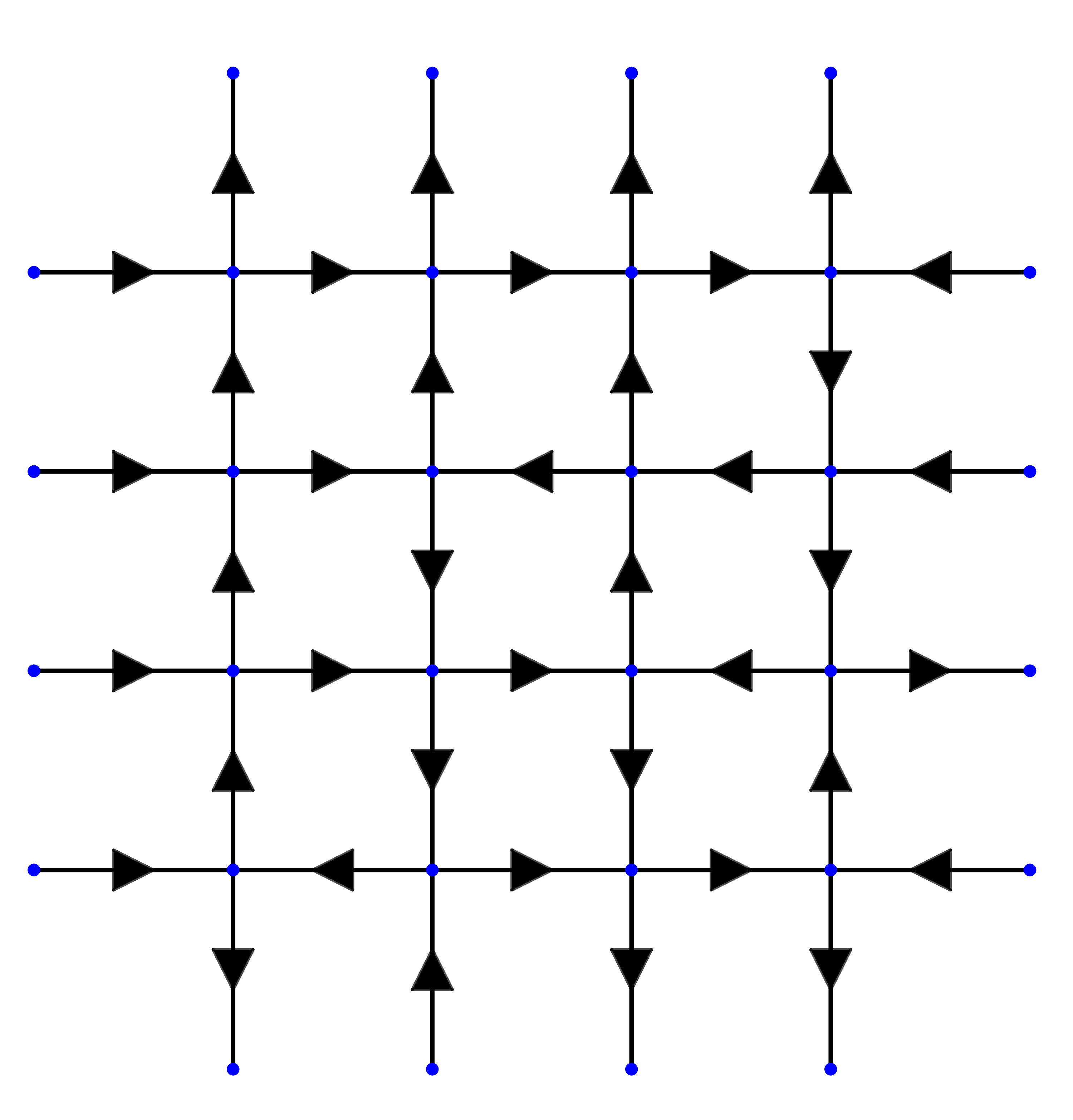}

\vspace{0.25in}

\includegraphics[width=2in]{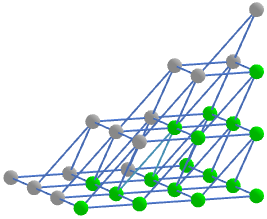} \hspace{0.25in} \includegraphics[width=2in]{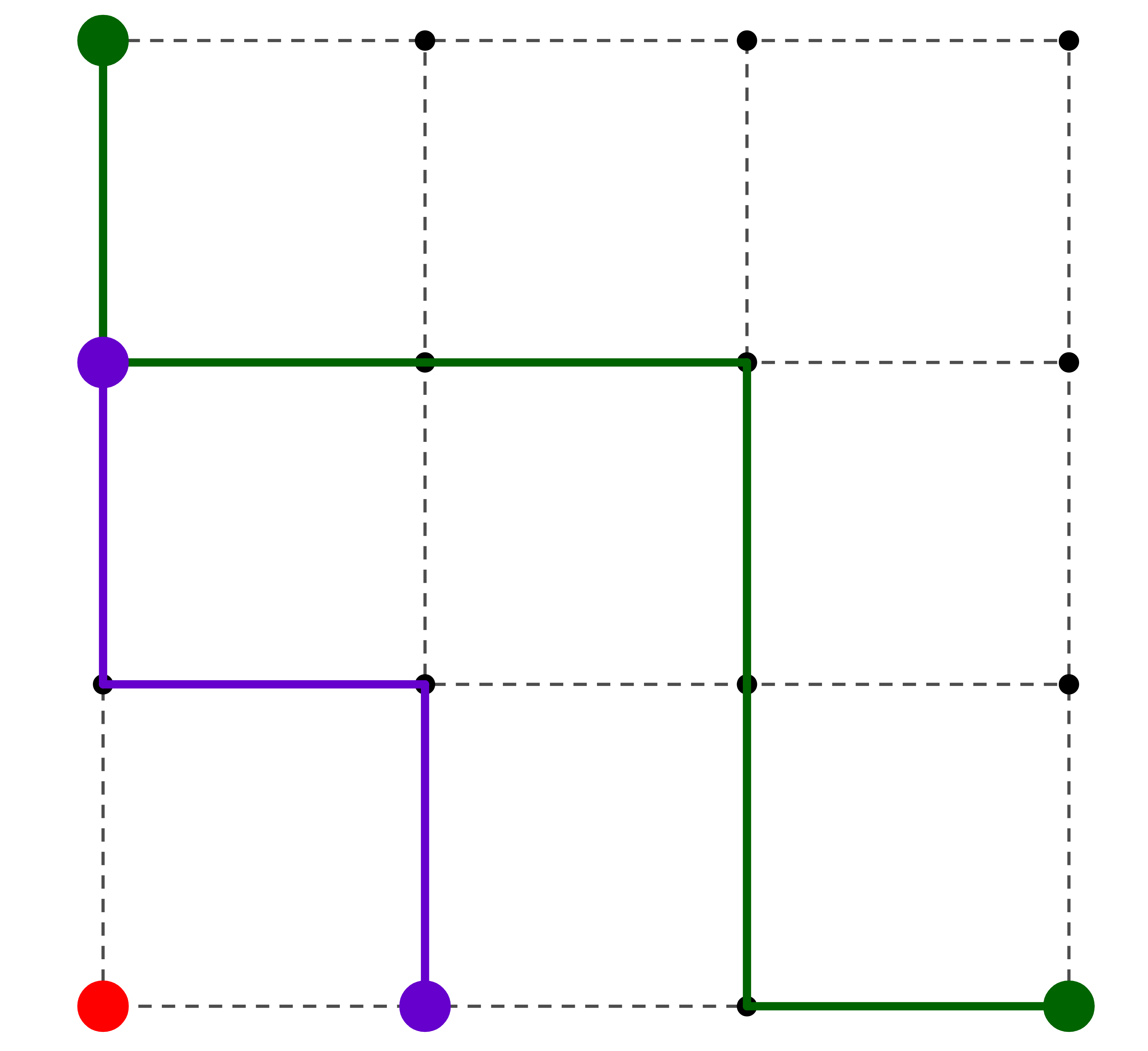}
\caption{From left to right, top to bottom: a $4 \times 4$ partial alternating sign matrix, along with its corresponding partial monotone triangle, partial height-function matrix, partial fully-packed loop configuration, rectangular ice configuration, order ideal, and nest of osculating paths.}
\label{allbij}
\end{figure}
Such bijections among combinatorial objects provide new angles from which to study them. For example, after these bijections, in Section~\ref{sec:pasm_dyn}, we apply an analogue of Wieland's gyration action \cite{Wieland} on fully-packed loop configurations, and following \cite{Striker1}, we relate this to the study of toggles and order ideals. We obtain the following theorem, showing that rowmotion on order ideals of a certain poset and gyration on partial fully-packed loop configurations have the same orbit structure.

\begin{EquivBij}
$J\left(\textbf{P}_{m,n}\right)$ under \text{Row} and $(m,n)$-partial fully-packed loop configurations under gyration are in equivariant bijection.
\end{EquivBij}

\section{Preliminaries}
\label{sec:prelim}
In this section we give known definitions and results which may help the reader to understand the main results in Sections~\ref{sec:pasm_bij} and \ref{sec:pasm_dyn}. As mentioned in the introduction, there exists many bijections between alternating sign matrices and other combinatorial objects. We will include a description of the bijection with monotone triangles in this section. See \cite{Propp} and \cite{Striker1} for details on the other bijections. Generalizations of these bijections are given in Section~\ref{sec:pasm_bij}. We begin with the definition of an alternating sign matrix.

\begin{definition}
An \emph{alternating sign matrix of order $n$} is an $n \times n$ matrix with entries in $\left\{-1, 0, 1 \right\}$ whose rows and columns sum to 1 and whose entries alternate in sign across each row and column. We will denote the set of all alternating sign matrices of order $n$ by $A_n$.
\end{definition}

Alternating sign matrices have a beautiful counting formula: $$ \left| A_n \right| = \prod_{j=0}^{n-1} \frac{(3j+1)!}{(n+j)!}.$$ This formula was conjectured in 1983 \cite{MRRASM}, and proved many years later \cite{Zeilberger, Kuperberg}.

We now give the definition of a monotone triangle and describe the bijection between monotone triangles and alternating sign matrices.

\begin{definition}
A \emph{monotone triangle of order $n$} is a triangular array of positive integers \\ $\left(a_{i,j}\right)_{1\leq i \leq j \leq n}$ taken from the set ${1,2,\ldots,n}$ with the following properties:

\begin{itemize}

\item rows are strictly increasing: $a_{i,j} < a_{i,j+1}$, and

\item diagonals are weakly increasing: $a_{i,j} \leq a_{i-1,j}$ and $a_{i,j} \leq a_{i+1,j+1}$.

\end{itemize}

\end{definition}

Given an alternating sign matrix, first construct its matrix of partial column sums. Then create the corresponding monotone triangle whose $i$th row consists of the values $j$ for which entry $(i,j)$ of this partial sum matrix is 1. An example is given in Figure~\ref{fig:asm}. 

\begin{figure}[hbt]
\centering
$\begin{pmatrix}
0 & 1 & 0 & 0 \\
1 & -1 & 0 & 1 \\
0 & 0 & 1 & 0 \\
0 & 1 & 0 & 0
\end{pmatrix}
\longleftrightarrow
\begin{pmatrix}
0 & 1 & 0 & 0 \\
1 & 0 & 0 & 1 \\
1 & 0 & 1 & 1 \\
1 & 1 & 1 & 1
\end{pmatrix}
\longleftrightarrow
\begin{array}{ccccccc}
  &   &   & 2 &   &   &\\
  &   & 1 &   & 4 &   & \\
  & 1 &   & 3 &   & 4 & \\
1 &   & 2 &   & 3 &   & 4
\end{array}$
\caption{A $4 \times 4$ alternating sign matrix along with its matrix of partial column sums and its corresponding monotone triangle.}
\label{fig:asm}
\end{figure}

Another of the objects alternating sign matrices are in bijection with are fully-packed loop configurations (or just fully-packed loops), which we define below. In order to do so, we must first define a certain graph.

\begin{definition}\label{gridmn}
Define the graph $G_{m,n}$ as follows. The vertex set is: $$V_{m,n} \defeq \left\{v_{i,j} \; : \; 0 \leq i \leq m+1, \; 0 \leq j \leq n+1 \right\} - \left\{v_{0,0}, v_{0, n+1}, v_{m+1, 0}, v_{m+1, n+1}\right\}.$$ We say the \emph{internal vertices} are $\left\{v_{i,j} \; : \; 1 \leq i \leq m, \; 1 \leq j \leq n\right\}$, and the remaining vertices are \emph{boundary vertices}. We also say a vertex $v_{i,j}$ is \emph{even} (resp. \emph{odd}) if $i+j$ is even (resp. odd). The edge set is: \[E_{m,n} \defeq \begin{cases} v_{i,j}v_{i+1,j} & 0 \leq i \leq m, 1 \leq j \leq n\\ v_{i,j}v_{i,j+1} & 1 \leq i \leq m, 0 \leq j \leq n. \end{cases}\]
\end{definition}

\noindent See Figure~\ref{fig:gmn} for a visual example of $G_{m,n}$.

\begin{definition}
\label{FPL}
A \emph{fully-packed loop configuration} of order $n$ is a subgraph of $G_{n,n}$ such that each interior vertex has exactly two incident edges and the following boundary conditions are met. When $n$ is odd, edges include $v_{0,j}v_{1,j}$ and $v_{n,j}v_{n+1,j}$ for $j$ odd, as well as $v_{i,0}v_{i,1}$ and $v_{i,n}v_{i,n+1}$ for $i$ even. When $n$ is even, edges include $v_{0,j}v_{1,j}$ for $j$ odd, $v_{n,j}v_{n+1,j}$ for $j$ even, $v_{i,0}v_{i,1}$ for $i$ even, and $v_{i,n}v_{i,n+1}$ for $i$ odd.
\end{definition} 
\noindent See Figure~\ref{fplfig} for examples of fully-packed loops.

\begin{figure}[hbt]
\centering
\includegraphics[scale=0.25]{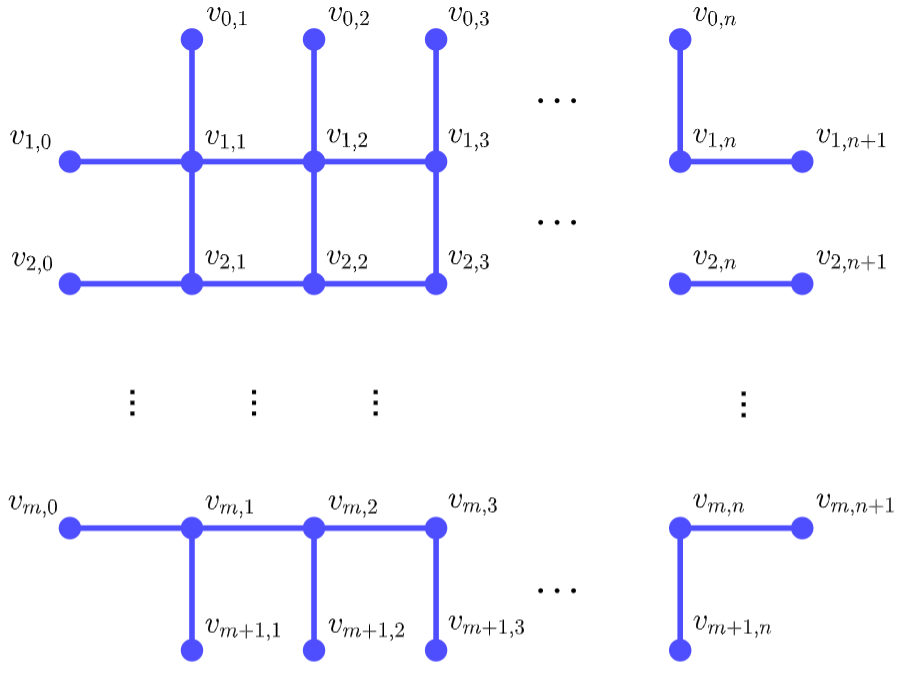}
\caption{The graph $G_{m,n}$.}
\label{fig:gmn}
\end{figure}

\begin{figure}[hbt]
\includegraphics[scale=0.4]{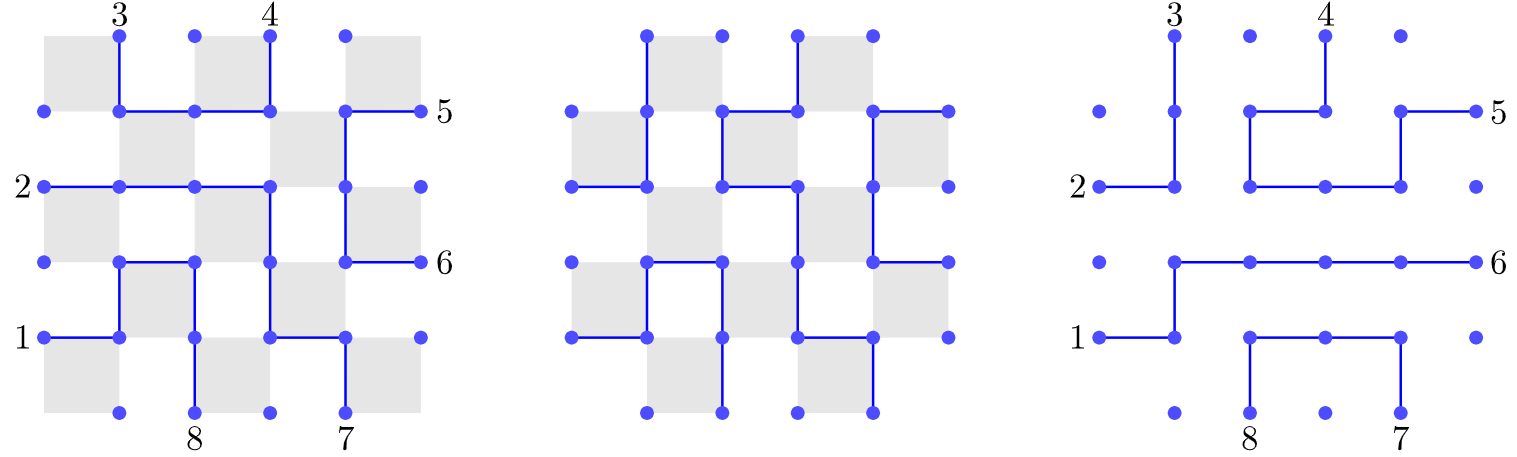}
\caption{The gyration action on a fully-packed loop configuration in $G_{4,4}$, first performing the local action on even squares (shaded on the left) and then odd squares (shaded in the middle). The initial and final fully-packed loops have labeled boundary edges for constructing their link patterns.}
\label{fplfig}
\end{figure}

We now introduce an action on fully-packed loops which has some very nice properties \cite{Wieland}. First, define a local action on a square of a  fully-packed loop as follows. If its  only edges are a pair of parallel lines, swap them to a pair of parallel lines in the other orientation. Otherwise, do nothing.

Also assign a parity to each square, starting with even in the top left corner, and alternating between even and odd for adjacent squares. Then define an action on a fully-packed loop as applying this local move first to all of the even squares, and then all of the odd squares. This produces another fully-packed loop since the local action does not change the degree of any interior vertex and does not affect the boundary conditions. Call this action \emph{gyration}. An example of this action can be seen in Figure~\ref{fplfig}.

One can label the vertices along the boundary of a fully-packed loop where the edges ``exit'' the graph with the numbers $\{1,\ldots,2n\}$. Then each number is connected to one other by a path, and each fully-packed loop can be reduced to a \emph{non-crossing matching}, meaning that if the numbers are arranged in a circle and connected with arcs, none of the arcs cross. This is called the \emph{link pattern} of a fully-packed loop. Wieland showed the following theorem \cite{Wieland}.

\begin{theorem}[\protect{\cite{Wieland}}]
\label{rotate}
Gyration on a fully-packed loop rotates the corresponding link pattern.  Specifically, if $i$ and $j$ are connected in a link pattern, then in the corresponding link pattern after gyration is applied, $i-1$ and $j-1$ (mod $2n$) will be connected.
\end{theorem}
\noindent See Figure~\ref{rotation} for an example of this rotation.

\begin{figure}[hbtp]
\centering
\includegraphics[scale=0.6]{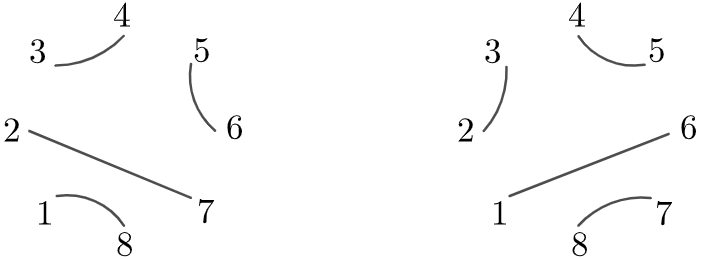}
\caption{The link patterns for the fully-packed loops from Figure~\protect{\ref{fplfig}}.}
\label{rotation}
\end{figure}

Some combinatorial objects behave particularly nicely with respect to various actions --- notably those which break up the set of objects into \emph{orbits}, such as when the action is bijective and the set is finite. Given a set of objects and such an action, it is especially noteworthy when the \emph{order} of the action is predictable and when the action exhibits special properties, such as the \emph{cyclic sieving phenomenon} \cite{CSP} or \emph{homomesy phenomenon} \cite{ProppRoby}. For more information on the background and an overview of some results in this area called \emph{dynamical algebraic combinatorics}, see \cite{StrikerNotices}. For more detailed information on results related to homomesy, see, for example \cite{Roby}.

While much could be said about this topic, we now provide a few relevant definitions which will help the reader understand the results of Section~\ref{sec:pasm_dyn}.

\begin{definition}
\label{def:poset}
A \emph{partially ordered set}, or \emph{poset}, is a set $P$ with a binary relation $\leq$ which is reflexive, antisymmetric, and transitive. We say $y$ \emph{covers} $x$ in $P$ (or $x$ is covered by $y$) if $x<y$ and no element $z$ exists in $P$ such that $x<z<y$. A poset is often represented by its \emph{Hasse diagram}, a graph whose vertices are the elements of $P$ and whose edges are the covering relations; if $x<y$, then $y$ is drawn above $x$ in the Hasse diagram.
\end{definition}

For more information about partially ordered sets, and for additional relevant terminology, see Chapter 3 of \cite{stanley1}.

\begin{definition}
\label{def:oi}
Let $P$ be a poset. An $\emph{order ideal}$ of $P$ is a subset $X \subseteq P$ such that if $y \in X$ and $z \leq y$, then $z \in X$. The set of order ideals of $P$ is denoted $J(P)$.
\end{definition}
\noindent $J(P)$, ordered by inclusion, is itself a poset. In fact it is a \emph{distributive lattice} (see Section 3.4 of \cite{stanley1}).

We now define an action on order ideals of poset. This action was first introduced for hypergraphs in 1974 by Duchet \cite{Duchet} and generalized to general posets by Brouwer and Schrijver in 1975 by Brouwer and Schrijver \cite{brouwer-schrijver}. It was studied in detail by Cameron and Fon-Der-Flaass for two chains and three chains \cite{fonder1, Cameron-Flaass}, and studied later by many others, including Striker and Williams \cite{ProRow}. This action continues to be of intense current interest.

\begin{definition}
\label{def:toggle}
Let $P$ be a poset, and let $X \in J(P)$. For any $q \in P$, the \emph{toggle} $t_q: J(P)\rightarrow J(P)$ is defined as follows:
$$t_q(X) = \begin{cases}
X \cup \{q\} & \text{if } q\notin X \text{ and } X \cup \{q\} \in J(P) \\
X - \{q\} & \text{if } q \in X \text{ and } X - \{q\} \in J(P) \\
X & \text{otherwise.}
\end{cases}$$
The \emph{toggle group} of $P$, denoted $T(P)$, is the group generated by the $t_q$ for all $q \in P$, with operation composition.
\end{definition}

In other words, if an element is not in the order ideal, and adding it in would result in another order ideal, then toggling the element adds it in. If an element is in the order ideal, and removing it would result in another order ideal, then toggling the element removes it. Otherwise, toggling the element does nothing. Notice that toggles are involutions, and that toggles commute whenever there are no covering relations between the elements \cite[p.546]{Cameron-Flaass}.

\begin{definition}
Let $P$ be a poset, and let $X \in J(P)$. Then \emph{rowmotion}, $\text{Row}(X)$, is the order ideal generated by the minimal elements of $P-X$.
\end{definition}

In \cite{Cameron-Flaass}, it was shown that rowmotion on $X$ is the order ideal obtained from $X$ by toggling the elements of $P$ from top to bottom (e.g., toggling by rows). The order of rowmotion on order ideals of particular posets, such as \emph{root posets} and \emph{minuscule posets} exhibits very nice properties, such as having predictable order, cyclic sieving, and homomesy. In some cases, rowmotion corresponds to another action, called promotion, on \emph{tableaux} or tableaux-like objects \cite{DilksPechenikStriker, DilksStrikerVorland, ProRow}. Interpreting actions as toggle group actions on order ideals of certain posets can be advantageous. There may be results or properties of the original action that reveal something new about the posets, or conversely, known results of posets may be used to prove something about the original action.

\section{Partial Alternating Sign Matrix Bijections}
\label{sec:pasm_bij}
In this section, we first define \emph{partial alternating sign matrices}. We then describe several bijections akin to those in the usual alternating sign matrix setting (see, for example \cite{Propp, Striker1}). These bijections are analogous to those in the usual alternating sign matrix setting, but we include proofs for each of them for completeness and clarity.

\begin{definition}
\label{pasm}
An \emph{$m \times n$ partial alternating sign matrix} is an $m \times n$ matrix $M = \left(M_{ij}\right)$ with entries in $\left\{-1,0,1\right\}$ such that:
\begin{itemize}
\item the entries in each row and column sum to either 0 or 1,
\item the nonzero entries in each row and column alternate in sign, and
\item the first nonzero entry (if any) in each column and last nonzero entry (if any) in each row are~1.
\end{itemize}
We denote the set of all $m \times n$ partial alternating sign matrices as $\pasm_{m,n}$. 

$\pasm_{2,2}$ consists of the following eight matrices:
\[\left(\begin{array}{rr} 0 & 0 \\ 0 & 0 \end{array}\right), \left(\begin{array}{rr} 1 & 0 \\ 0 & 0 \end{array}\right), \left(\begin{array}{rr} 0 & 1 \\ 0 & 0 \end{array}\right), \left(\begin{array}{rr} 0 & 0 \\ 1 & 0 \end{array}\right), \left(\begin{array}{rr} 0 & 0 \\ 0 & 1 \end{array}\right), \left(\begin{array}{rr} 1 & 0 \\ 0 & 1 \end{array}\right), \left(\begin{array}{rr} 0 & 1 \\ 1 & 0 \end{array}\right), \text{ and } \left(\begin{array}{rr} 1 & 0 \\ -1 & 1 \end{array}\right).\]

$\pasm_{2,3}$ consists of the following seventeen matricies:
\[ \left(\begin{array}{rrr} 0 & 0 & 0 \\ 0 & 0 & 0 \end{array}\right), \left(\begin{array}{rrr} 1 & 0 & 0 \\ 0 & 0 & 0 \end{array}\right), \left(\begin{array}{rrr} 0 & 1 & 0 \\ 0 & 0 & 0 \end{array}\right), \left(\begin{array}{rrr} 0 & 0 & 1 \\ 0 & 0 & 0 \end{array}\right), \left(\begin{array}{rrr} 0 & 0 & 0 \\ 1 & 0 & 0 \end{array}\right), \left(\begin{array}{rrr} 0 & 0 & 0 \\ 0 & 1 & 0 \end{array}\right),\]
\[\left(\begin{array}{rrr} 0 & 0 & 0 \\ 0 & 0 & 1 \end{array}\right), \left(\begin{array}{rrr} 1 & 0 & 0 \\ 0 & 1 & 0 \end{array}\right), \left(\begin{array}{rrr} 1 & 0 & 0 \\ 0 & 0 & 1 \end{array}\right), \left(\begin{array}{rrr} 0 & 1 & 0 \\ 1 & 0 & 0 \end{array}\right), \left(\begin{array}{rrr} 0 & 1 & 0 \\ 0 & 0 & 1 \end{array}\right), \left(\begin{array}{rrr} 0 & 0 & 1 \\ 1 & 0 & 0 \end{array}\right),\]
\[\left(\begin{array}{rrr} 0 & 0 & 1 \\ 0 & 1 & 0 \end{array}\right), \left(\begin{array}{rrr} 1 & 0 & 0 \\ -1 & 1 & 0 \end{array}\right) \left(\begin{array}{rrr} 1 & 0 & 0 \\ -1 & 0 & 1 \end{array}\right), \left(\begin{array}{rrr} 0 & 1 & 0 \\ 0 & -1 & 1 \end{array}\right), \text{ and } \left(\begin{array}{rrr} 0 & 1 & 0 \\ 1 & -1 & 1 \end{array}\right).\]
\end{definition}

The cardinality of $\pasm_{n,n}$ is given by sequence A202751 in the Online Encyclopedia of Integer Sequences \cite{oeis1}. It is unlikely that there exists a product formula for $|\pasm_{m,n}|$, since, for example, $|\pasm_{6,6}|=1442764=2^2 \cdot 373 \cdot 967$. See Figure~\ref{pasm_table} for the number of $m \times n$ partial alternating sign matrices for $m,n \leq 6$, calculated with SageMath~\cite{sage}. Note that the number of $m \times n$ partial alternating sign matrices is the same as the number of those of size $n \times m$. To transform $M \in \pasm_{m,n}$ into $M' \in \pasm_{n,m}$, take the $i$th row of $M$, reverse the order of the entries, and use that as column $m-i+1$ in $M'$. For example, under this mapping the matrix $\left(\begin{array}{rrr}a & b & c \\ d & e & f \end{array}\right)$ would correspond to the matrix $\left(\begin{array}{rr} f & c \\ e & b \\ d & a \end{array}\right)$

\begin{figure}[hbtp]
\centering
\begin{tabular}{|c|c|c|c|c|c|c|c|}
\hline
\diaghead(1,-1)%
   {\theadfont nnn}%
   {$m$}{$n$} & 1 & 2 & 3 & 4 & 5 & 6 \\ \hline
1 & 2 & 3  & 4    & 5     & 6      & 7 \\ \hline
2 & 3 & 8  & 17   & 31    & 51     & 78 \\ \hline
3 & 4 & 17 & 62   & 184   & 462    & 1022 \\ \hline
4 & 5 & 31 & 184  & 924   & 3809   & 13197 \\ \hline
5 & 6 & 51 & 462  & 3809  & 26394  & 150777 \\ \hline
6 & 7 & 78 & 1022 & 13197 & 150777 & 1442764 \\ \hline
\end{tabular}
\caption{A table of the number of $m \times n$ partial alternating sign matrices (bottom).}
\label{pasm_table}
\end{figure}

\begin{remark} 
\label{rmk:fortin}
Partial alternating sign matrices can be viewed as a generalization of partial permutation matrices, which are matrices whose entries are in $\{0,1\}$ and which have at most one $1$ in any given row and column. $n \times n$ partial permutation and alternating sign matrices were studied in a different context by Fortin~\cite{Fortin}. He showed that, with a poset structure analogous to the strong Bruhat order, the lattice of partial alternating sign matrices is the \emph{MacNeille completion} of the poset of partial permutations. That is, the lattice of partial alternating sign matrices is the smallest lattice containing the poset of partial permutations. This is analogous to the result of Lascoux and Sch\"{u}tzenberger~\cite[Lemma 5.4]{Lascoux} that the lattice of $n\times n$ alternating sign matrices is the MacNeille completion of the strong Bruhat order on $S_n$.
\end{remark}

We now state our main result of this section, which is the goal of the remainder of this section.

\begin{theorem}\label{thm:bijections}
There are explicit bijections between each of the following:
\begin{itemize}
\item $m \times n$ partial alternating sign matrices, 
\item $(m,n)$-partial monotone triangles, 
\item $(m,n)$-partial height-function matrices, 
\item $(m,n)$-partial fully-packed loop configurations, 
\item $(m,n)$-rectangular ice configurations, 
\item the set of order ideals of $\textbf{P}_{m,n}$, and 
\item $(m,n)$-nests of osculating paths.
\end{itemize}
\end{theorem}

We prove this theorem via a series of lemmas and define the needed objects along the way. We begin with partial monotone triangles.

\begin{definition}\label{pmt}

An \emph{$(m,n)$-partial monotone triangle} is a triangular array of \textbf{non-negative} integers $\left(a_{i,j}\right)_{1 \leq j \leq i \leq m}$ with entries in $\{0,1,\ldots,n\}$ and with the following properties:

\begin{itemize}

\item rows are \emph{weakly} increasing: $a_{i,j} \leq a_{i,j+1}$,

\item \emph{nonzero} entries in rows are strictly increasing, and

\item diagonals are weakly increasing: $a_{i,j} \leq a_{i-1,j}$ and $a_{i,j} \leq a_{i+1,j+1}$.

\end{itemize}

\end{definition}

An example of a triangular array with indexing as described above is given in Figure~\ref{tri}, and an example of a partial monotone triangle can be found in Figure~\ref{examplesfig}

\begin{remark}
Objects similar to this definition have appeared in the literature before. In particular, a certain transformation of an $(n,n)$-partial monotone triangle is what is referred to as a ``generalized key'' in Section 5.2 of \cite{Fortin}.
\end{remark}

\begin{figure}[hbtp]
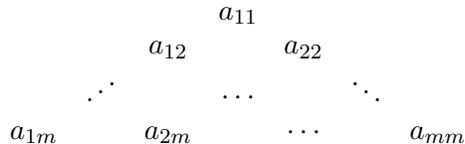

\centering
\[
\begin{array}{cccccccccc}
& & & a_{11} & & &\\
& & a_{12} & & a_{22}\\
& \reflectbox{$\ddots$} & & \cdots & & \ddots &\\
a_{1m} & & a_{2m} & & \cdots & & a_{mm}
\end{array}
\]
\caption{A triangular array of numbers with the indexing as described in Definition~\ref{pmt}.}
\label{tri}
\end{figure}

\begin{figure}[hbtp]
\centering
\[
\left(\begin{array}{rrrr}
1  & 0 & 0  & 0\\
0  & 0 & 1  & 0\\
-1 & 1 & 0  & 0\\
1  & 0 & -1 & 1
\end{array}\right)
\hspace{.2in}
\begin{array}{ccccccc}
  &   &   & 1 &   &   &  \\
  &   & 1 &   & 3 &   &  \\
  & 0 &   & 2 &   & 3 &  \\
0 &   & 1 &   & 2 &   & 4
\end{array}
\hspace{.2in}
\begin{pmatrix}
0 & 0 & 0 & 0 & 0\\
1 & 0 & 0 & 0 & 0\\
2 & 1 & 1 & 0 & 0\\
2 & 2 & 1 & 0 & 0\\
3 & 2 & 1 & 1 & 0
\end{pmatrix}
\hspace{.2in}
\begin{pmatrix}
0 & 1 & 2 & 3 & 4\\
1 & 2 & 3 & 4 & 3\\
2 & 3 & 2 & 3 & 2\\
3 & 4 & 3 & 2 & 3\\
4 & 3 & 4 & 3 & 2
\end{pmatrix}
\]
\caption{A $4 \times 4$ partial alternating sign matrix along with its corresponding partial monotone triangle, corner-sum matrix, and partial height-function matrix.}
\label{examplesfig}
\end{figure}

\begin{lemma}\label{pasm-pmt}
There is an explicit bijection between $m \times n$ partial alternating sign matrices and $(m,n)$-partial monotone triangles.
\end{lemma}

\begin{proof}
Given an $m \times n$ partial alternating sign matrix $M$, construct $(C_{i,j})_{1 \leq i \leq m, 1 \leq j \leq n}$, its $m \times n$ matrix of partial column sums, by setting $C_{ij} = \sum_{k=1}^i M_{kj}$. By the alternating condition on partial alternating sign matrices, this will be a $\{0,1\}$-matrix. Construct an array of numbers from $(C_{i,j})$ as follows: in the $i$th row, record in increasing order the values $j$ for which $C_{ij}=1$. If there are less than $i$ such values, fill in zeros from the left until there are $i$ values in row $i$. By construction, the nonzero entries in each row are strictly increasing. The restrictions on partial alternating sign matrices, namely the rows and columns summing to $0$ or $1$, the alternating condition, and the fact that the first (last) nonzero entry in each column (row) must be $1$ guarantee that the diagonals are weakly increasing from southwest to northeast and northwest to southeast. Thus, the result is an $(m,n)$-partial monotone triangle. Since any matrix can be uniquely determined by its partial column sums, this is a one-to-one map.

Furthermore, this map is easily reversible. Given an $(m,n)$-partial monotone triangle $\left(a_{i,j}\right)$, first build an $m \times n$ $\{0,1\}$-matrix $(C_{i,j})$ by setting $C_{i,k} = 1$ whenever $a_{i,j} = k > 0$ and filling the rest of the matrix with zeros. This matrix $(C_{i,j})$ records the partial column sums of a unique partial alternating sign matrix, $M$, where $M_{1j}=C_{1j}$ and $M_{ij} = C_{ij}-C_{i-1,j}$ for $2 \leq i \leq m$.
\end{proof}

Putting the natural partial order (component-wise comparison) on partial monotone triangles forms a lattice, which is the MacNeille completion of the poset of partial permutations as mentioned in Remark~\ref{rmk:fortin}. This is proven for $m=n$ in \cite{Fortin}, and can be extended to $m \neq n$. A similar poset construction using all ``sign matrices'' was studied in \cite{Brualdi-Dahl}.

\begin{figure}[ht]
\centering
\begin{minipage}{.45\textwidth}
\centering
\includegraphics[scale=0.4]{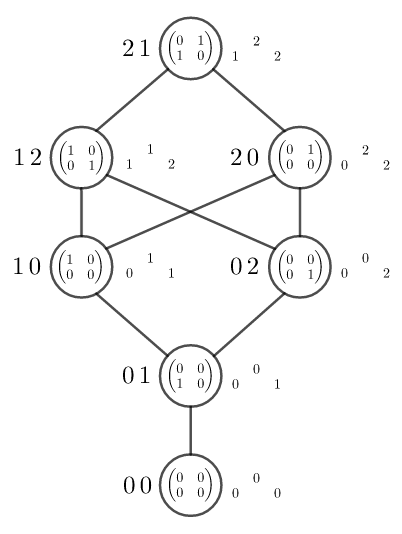}
\end{minipage}
\begin{minipage}{.45\textwidth}
\centering
\includegraphics[scale=0.33]{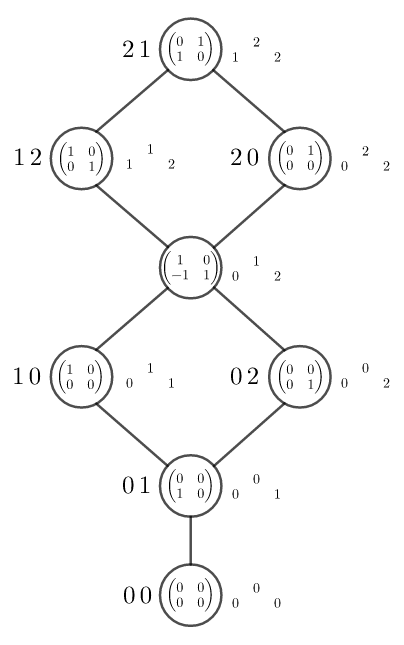}
\end{minipage}
\caption{The poset of partial permutations (left) and the poset of partial alternating sign matrices (right), for $m=n=2$.}
\label{posets}
\end{figure}

We now move on to \emph{corner-sum matrices}, which are useful for showing the correspondence with \emph{partial height-function matrices}.

\begin{definition}\label{csm}
Given any $m \times n$ matrix $\left(M_{i,j}\right)_{1\leq i \leq m, 1 \leq j \leq n}$, we can define its (north-east) \emph{corner-sum matrix} $\left(c_{i,j}\right)_{0\leq i \leq m, 0 \leq j \leq n}$ by setting $c_{i,j} = \displaystyle\sum_{\substack{i'\geq i \\ j'\leq j}} M_{i',j'}$.
\end{definition}

\begin{lemma}\label{pasm-csm}
There is an explicit bijection between $m \times n$ partial alternating sign matrices and $(m+1) \times (n+1)$ matrices whose first row and last column are all zeros, and whose entries are increasing by from top to bottom and right to left and whose entries increase by at most one.
\end{lemma}

\begin{proof}
We claim the set of matrices described above is exactly the set of corner-sum matrices of $m \times n$ partial alternating sign matrices. The first row and last column of any such corner-sum matrix are all zeros by definition. The fact that an $m \times n$ partial alternating matrix has entries in $\{-1,0,1\}$, along with the alternating condition guarantees that adjacent entries in the corresponding corner-sum matrix differ by at most one. Note that given any corner-sum matrix $c$, one can recover the original matrix $M$ by $M_{ij}=c_{ij}-c_{i-1,j}-c_{i,j+1}+c_{i-1,j+1}$.
\end{proof}

See Figure~\ref{examplesfig} for an example of a partial alternating sign matrix along with its corresponding corner-sum matrix.

\begin{definition}\label{phf}

An \emph{(m,n)-partial height-function matrix} is a matrix $\left(h_{i,j}\right)_{0\leq i \leq m, 1 \leq j \leq n}$ with the following properties:
\begin{itemize}
\item all entries are non-negative,
\item $h_{0,k} = k$ for $0 \leq k \leq n$, and $h_{\ell,0} = \ell$ for $0 \leq \ell \leq m$, and
\item the difference between any pair of row-adjacent or column-adjacent entries is $\pm 1$.
\end{itemize}

\end{definition}

\begin{lemma}\label{pasm-phf}
There is an explicit bijection between $m \times n$ partial alternating sign matrices and $(m,n)$-partial height-function matrices.
\end{lemma}

\begin{proof}
The bijection is given by first mapping a partial alternating sign matrix $\left(M_{i,j}\right)_{1 \leq i \leq m, 1 \leq j, \leq n}$ to its corner-sum matrix $(c_{i,j})$. By Lemma~\ref{pasm-csm}, this is an $(m+1) \times (n+1)$ matrices whose first row and last column are all zeros, and whose entries increase by at most one from top to bottom and right to left. Then define the matrix $\left(h_{i,j}\right)_{0\leq i \leq m, 0 \leq j, \leq n}$ as $h_{i,j} = i+j-2c_{i,n-j}.$ $(h_{i,j})$ satisfies Definition~\ref{phf} because of the following. The condition that the first row and last column of the corner-sum matrix are $0$ guarantees that the first row and first column start at 0 and increase by 1. The condition that the entries of the corner-sum matrix increase by at most one from top to bottom and right to left is equivalent to the condition that adjacent entries in $(h_{i,j})$ differ by exactly one. In particular, when moving along a row (right-to-left) or column (top-to-bottom), if the entry stays the same in the corner-sum matrix, then the corresponding entry in $(h_{i,j})$ goes down by one, and if the entry increases by one in the corner-sum matrix, then the corresponding entry in $(h_{i,j})$ goes up by one (this can be seen using the formula $h_{i,j} = i+j-2c_{i,n-j}$). Thus $(h_{i,j})$ is a partial height-function matrix.

Given a partial height-function matrix $(h_{i,j})$, we can recover $(c_{i,j})$ by setting $c_{i,n-j}=\frac{i+j-h_{i,j}}{2}$, from which we can recover the partial alternating sign matrix as mentioned in the proof of Lemma~\ref{pasm-csm}.
\end{proof}

An example of a partial alternating sign matrix and its corresponding partial height-function matrix is given in Figure~\ref{examplesfig}.

The next several bijections make use of the graph $G_{m,n}$ of Definition~\ref{gridmn}. These bijections will give us more interesting angles from which to study partial alternating sign matrices, with visualizable objects and connections to physics and dynamics.

\begin{definition}\label{pfpl}
An \emph{$(m,n)$-partial fully-packed loop configuration} is a subgraph of $G_{m,n}$ whose edges include $v_{0,j}v_{1,j}$ for $j$ odd and $v_{i,0}v_{i,1}$ for $i$ even, and such that each interior vertex has exactly two incident edges. 
\end{definition}

See Figure~\ref{fig:ice-fpl}, left, for an example.

\begin{figure}[hbtp]
\centering
\includegraphics[width=1.75in]{fplnew.png} \hspace{0.5in} \includegraphics[width=1.75in]{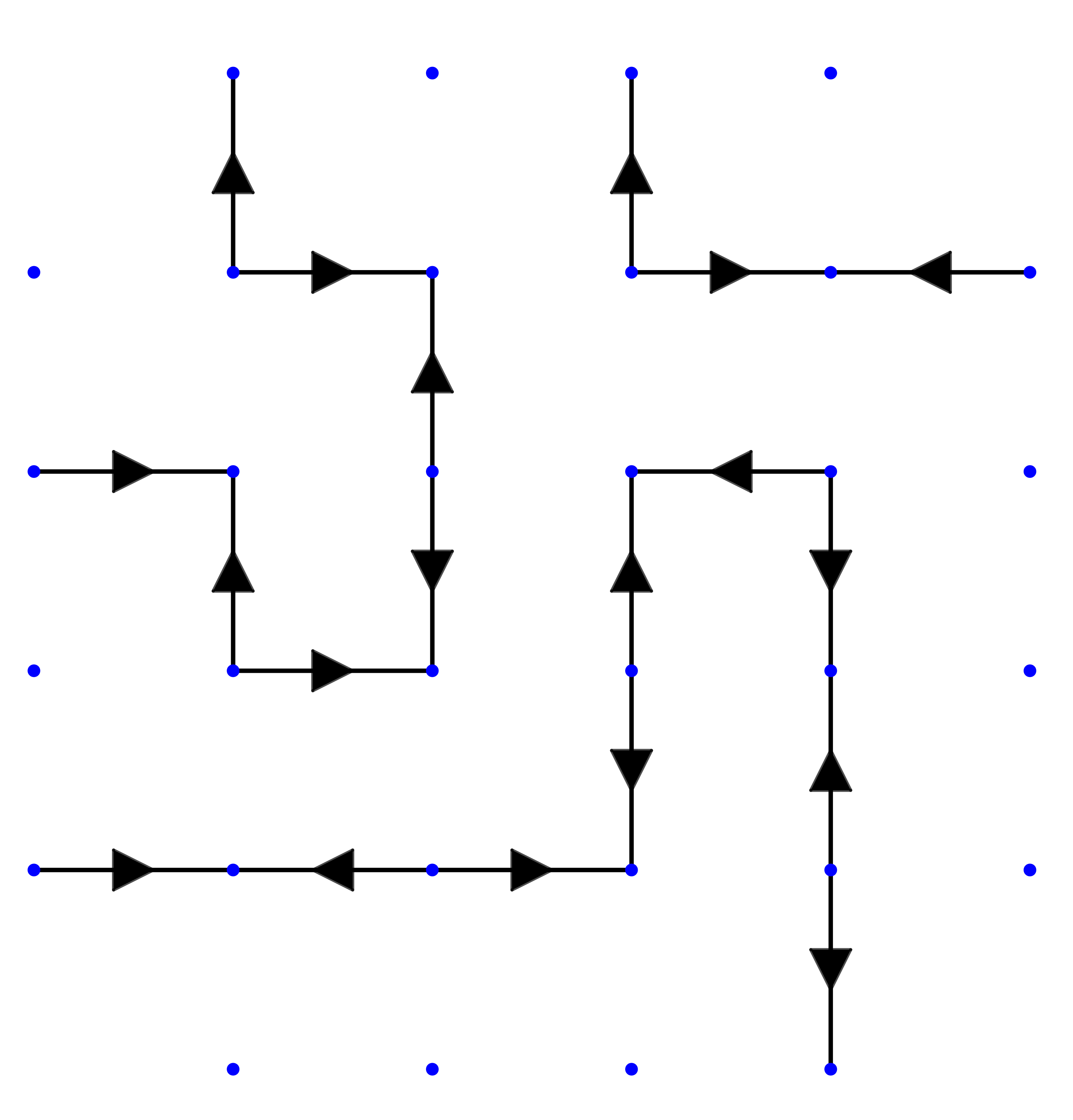} \hspace{0.5in} \includegraphics[width=1.75in]{rectangularice-update.png}
\caption{An example of a $(4,4)$-partial fully-packed loop configuration, this same configuration with its edges directed, and the associated $(4,4)$-rectangular ice configuration corresponding to the partial alternating sign matrix from Figure~\ref{examplesfig}.}
\label{fig:ice-fpl}
\end{figure}

\begin{lemma}\label{pfpl-phf}
There is an explicit bijection between $(m,n)$-partial fully-packed loop configurations and $(m,n)$-partial height-function matrices.
\end{lemma}

\begin{proof}
See Figure~\ref{fpl-ex} for an example of this bijection. Starting with an $(m,n)$-partial height-function matrix $(h_{i,j})$, overlay $G_{m,n}$ so that each interior number in the partial height-function matrix has four surrounding vertices. Separate two horizontally adjacent numbers by an edge if the numbers are $2k$ and $2k+1$ (in either order) for any integer $k$. Separate two vertically adjacent numbers by an edge if the numbers are $2k-1$ and $2k$ (in either order) for any integer $k$. Since $h_{0,k} = k$ for $0 \leq k \leq n$, and $h_{\ell,0} = \ell$ for $0 \leq \ell \leq m$, we will have the required boundary conditions for an $(m,n)$-partial fully-packed loop. Also, each partial height-function matrix entry differs by exactly one, so this yields the condition that interior vertices have exactly two incident edges. For the reverse map, start with the partial fully-packed loop configuration and fill in the boundary conditions for the partial height-function matrix. Then, simply use the rule described above in reverse, using the condition that each adjacent entry in the partial height-function matrix must be exactly one apart.
\end{proof}

\begin{figure}[ht]
\centering
\includegraphics[scale=0.25]{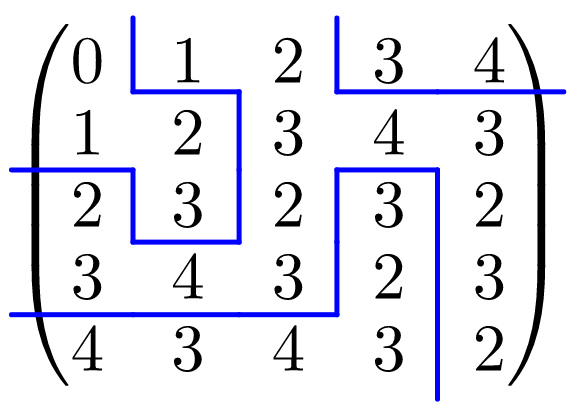}
\caption{The partial height-function matrix from Figure~\ref{examplesfig} overlaid with its corresponding partial fully-packed loop configuration.}
\label{fpl-ex}
\end{figure}

\begin{definition}\label{ice}
An \emph{$(m,n)$-rectangular ice configuration} is a directed graph whose underlying graph is $G_{m,n}$ such that the directed edges \textbf{along the left} point inward, the directed edges \textbf{on the top} point outward, and each interior vertex has both in-degree and out-degree equal to two.
\end{definition}

An example of a rectangular ice configuration is given in Figure~\ref{fig:ice-fpl}, right. In general, when a directed graph whose underlying graph is $G_{m,n}$ has all directed edges pointing inward along the left \emph{and right} and outward along the top \emph{and bottom}, that graph is said to have \emph{domain wall boundary conditions}. Here, we have a partial analogue of those conditions. When $m=n$ and domain wall boundary conditions are satisfied, we have \emph{square ice configurations}, which are in bijection with usual alternating sign matrices.

\begin{lemma}\label{pfpl-ice}
There is an explicit bijection between $(m,n)$-rectangular ice configurations and $(m,n)$-partial fully-packed loop configurations.
\end{lemma}

\begin{proof}
Given an $(m,n)$-rectangular ice configuration, keep only those edges which start at an even vertex and end at an odd vertex (as defined in Definition \ref{gridmn}). Note that there is no loss of information here, as each interior vertex had in-degree and out-degree equal to two. Then make those edges undirected. Again, we lose no information here because we know that each of these edges was originally directed towards an odd vertex. Each vertex will now have degree equal to two, resulting in an $(m,n)$-partial fully-packed loop configuration by definition. Note that since we started with edges along the left pointing inward and edges along the top pointing outward, keeping edges that were originally directed towards an odd vertex give us the fixed boundary conditions along  the top and left edges for a partial fully-packed loop.

For the reverse direction, we simply start with an $(m,n)$-partial fully-packed loop configuration ``undo'' each of the steps: direct the edges so that they go from even vertices to odd vertices. Then, fill in the remaining directed edges so that each interior vertex has in- and out-degree equal to 2. The boundary condition on rectangular ice that the directed edges along the left point inward and the directed edges on the top point outward is satisfied because we started with a partial fully-packed loop, which by definition has edges $v_{0,j}v_{1,j}$ for $j$ odd and $v_{i,0}v_{i,1}$ for $i$ even.
\end{proof}
\noindent See Figure \ref{fig:ice-fpl} for an example of this bijection.

\begin{remark}
\label{rem:corresp}
Note that this result, along with a horizontal reflection due to our conventions, coincides with the usual correspondence between alternating sign matrices and square ice configurations. This correspondence is as follows. At each vertex which comes from a 1 in the matrix, the arrows on the horizontal edges of the digraph both point inwards, and the arrows on the vertical edges both point outwards. At each vertex which comes from a -1 in the matrix, the arrows on the horizontal edges both point outwards and the arrows on the vertical edges point outwards. At each vertex which comes from a 0 in the matrix, there are one each of inward and outward-pointing vertical and horizontal edges.
\end{remark}

We now relate partial alternating sign matrices to order ideals of a certain poset (recall Definitions~\ref{def:poset} and \ref{def:oi}). This, along with the connection to partial fully-packed loop configurations lays the groundwork for the study of related dynamics in Section~\ref{sec:pasm_dyn}.

\begin{definition}\label{pyramidposet}

Define the poset elements $\textbf{P}_{m,n}$ as the coordinates $(i,j,k)$ in $\bbz^3$ such that $0 \leq k \leq m-1$, $k \leq j \leq n-1$, and $k \leq i \leq m-1$. Define the partial order via the following covering relations: $(i,j,k)$ covers $(i+1, j, k)$, $(i, j+1, k)$, $(i-1, j, k-1)$, and $(i, j-1, k-1)$ whenever these coordinates are poset elements. We sometimes refer to this poset as the \emph{PASM poset}.

\end{definition}

See Figure~\ref{fig:pyramid} for an example of $\textbf{P}_{m,n}$. This poset is somewhat related to one in the usual alternating sign matrix setting, sometimes denoted $\textbf{A}_n$ \cite{Striker1}. In that setting, each ``layer'' is a triangular poset instead of a rectangular one. For comparison, an example (with some coordinates shifted to fit our conventions) is given in Figure~\ref{fig:pyramid}. The construction of $\textbf{P}_{m,n}$ directly gives us the following proposition and corollary.

\begin{proposition}\label{posetprop}

$\textbf{P}_{m,n}$ is a ranked poset in which the rank of $(i,j,k)$ in $\textbf{P}_{m,n}$ equals $m+n-2-i-j+2k$. It has a unique minimal element $(m-1,n-1,0)$  which has rank $0$. The maximal elements are of the form $(i,i,i)$ and have rank $m+n-2$.

\end{proposition}

\begin{figure}[htbp]
\centering
\includegraphics[scale=0.5]{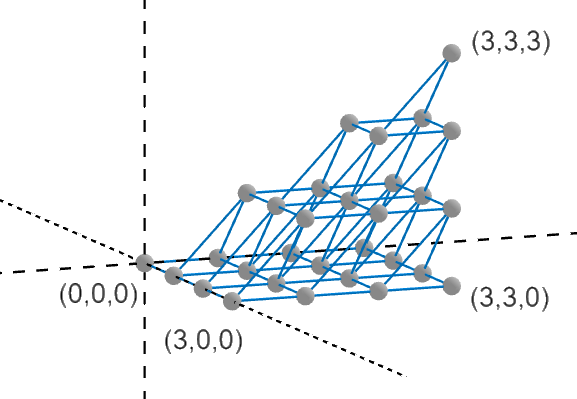}\hspace{1cm} \includegraphics[height=1.75in]{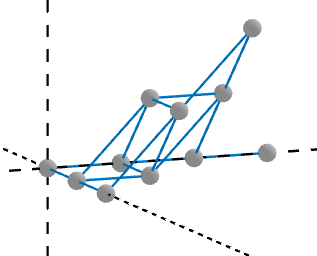}
\caption{An example of the PASM poset $\textbf{P}_{4,4}$ plotted in $\bbr^3$ (left) as well as the poset $\textbf{A}_4$ (right) for comparison. Elements are greater in $\textbf{P}_{4,4}$ moving northwest and up from the minimum element $(3,3,0)$.}
\label{fig:pyramid}
\end{figure}

\begin{corollary}
Let $a = \min(m,n)$ and $b=\max(m,n)$. Then $J(\textbf{P}_{m,n})$ is a distributive lattice of rank $\frac{-a^3+3a^2b+3ab+a}{6}$.
\end{corollary}

\begin{proof}Since we have an explicit construction of $\textbf{P}_{m,n}$, we can use this to determine the rank of $J(\textbf{P}_{m,n})$ (viewed as the lattice of order ideals by inclusion). The maximum element of $J(\textbf{P}_{m,n})$ has rank equal to the number of elements in $P_{m,n}$. At each $k$ value, we have $(a-k) \times (b-k)$ elements, and so we have $\displaystyle\sum_{k=0}^{a-1} (a-k)(b-k) = \frac{-a^3+3a^2b+3ab+a}{6}$ total elements.
\end{proof}

We now show that the order ideals of this poset are in correspondence with the other objects we have been studying.

\begin{lemma}\label{phf-oi}
There is an explicit bijection between $J(\textbf{P}_{m,n})$ and $(m,n)$-partial height-function matrices.
\end{lemma}

\begin{proof}
Let $\mathcal{O} \in J(\textbf{P}_{m,n})$. For fixed $i$ and $j$ with $0 \leq i \leq m-1$, $0 \leq j \leq n-1$, define the subset $S_{i,j} \defeq \left\{(i,j,t) \, | \, (i,j,t)\in\textbf{P}_{m,n}\right\}$. That is, $S_{i,j}$ is the intersection of the elements of $\textbf{P}_{m,n}$ from the explicit construction in Definition \ref{pyramidposet} with the vertical line $\left\{(i,j,t) : t \in \bbr\right\}$. Construct an $(m+1)\times(n+1)$ matrix $(h_{i,j})$ as follows. First, set $h_{0,k}=k$ for $0 \leq k \leq n$ and $h_{\ell,0}=\ell$ for $0 \leq \ell \leq m$. Then entry $h_{i,j}$ for $1 \leq i \leq m$, $1 \leq j \leq n$ is determined by the cardinality of the intersection between $\mathcal{O}$ and $S_{i,j}$. If this cardinality equals $k$, then, $h_{i,j} = i + j - 2k$. The fact that $\mathcal{O}$ is an order ideal, along with covering relations of $\textbf{P}_{m,n}$ guarantee that adjacent entries in $(h_{i,j})$ differ by exactly one, thus the result is an $(m,n$)-partial height-function matrix. Notice that if the cardinality of the intersection is $k$, then $h_{i,j}$ will be $2k$ less than the maximum possible value of that partial height-function entry.
\end{proof}
\noindent An example of this bijection is given in Figure~\ref{orderidealfig}.

\begin{figure}[ht]
\centering
$\begin{pmatrix}
0 & 1 & 2 & 3 & 4\\
1 & 2 & 3 & 4 & 3\\
2 & 3 & 2 & 3 & 2\\
3 & 4 & 3 & 2 & 3\\
4 & 3 & 4 & 3 & 2
\end{pmatrix}
\longleftrightarrow$\includegraphics[scale=0.6, valign=c]{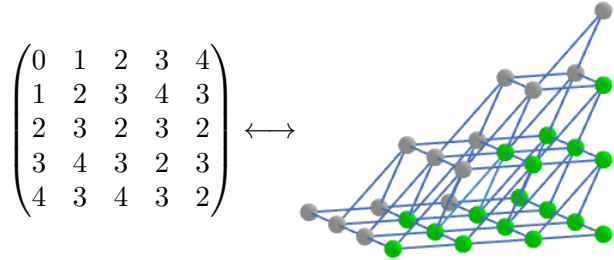}
\caption{A $(4,4)$-partial height-function matrix and its corresponding order ideal (shown in green) in $\textbf{P}_{4,4}$.}
\label{orderidealfig}
\end{figure}

Finally, we relate partial alternating sign matrices to certain sets of nested lattice paths. Here, \emph{lattice path} means a path in the $m \times n$ grid, consisting of south and east steps. When we say \emph{corner}, we mean a south step followed by an east step. This allows us to explicitly count the number of partial alternating sign matrices with total sum $1$, which gives some hope of finding a refined enumeration of partial alternating sign matrices.

\begin{definition}
An \emph{$(m,n)$-nest of osculating lattice paths} is any set of osculating paths (non-crossing paths which can only touch at corners) in the $m \times n$ grid with south and east steps, whose starting points are on the left side, and whose end points are on the bottom of the grid.
\end{definition}

See Figure~\ref{fig:paths} for an example. The next lemma follows from known results of osculating paths (see \cite{Behrend2} for more information and further references) and so we present it without proof. In Figure~\ref{fig:paths}, we give an example which goes through the bijections discussed earlier in this section, and in Figure~\ref{fig:directpaths} we give an example of a more standard and direct way (akin to that of the usual alternating sign matrix setting) to obtain the nest of paths.


\begin{lemma}\label{pasm-olp}
There is an explicit bijection between $m \times n$ partial alternating sign matrices with total sum $t$ and the set of $(m,n)$-nests of osculating lattice paths with $t$ paths.
\end{lemma}

\begin{figure}[hbtp]
\centering
\scalebox{1.25}{$\left(\begin{array}{rrrr}
1  & 0 & 0  & 0\\
0  & 0 & 1  & 0\\
-1 & 1 & 0  & 0\\
1  & 0 & -1 & 1
\end{array}\right)$}
\hspace{0.5in}\includegraphics[width=1.75in, valign=c]{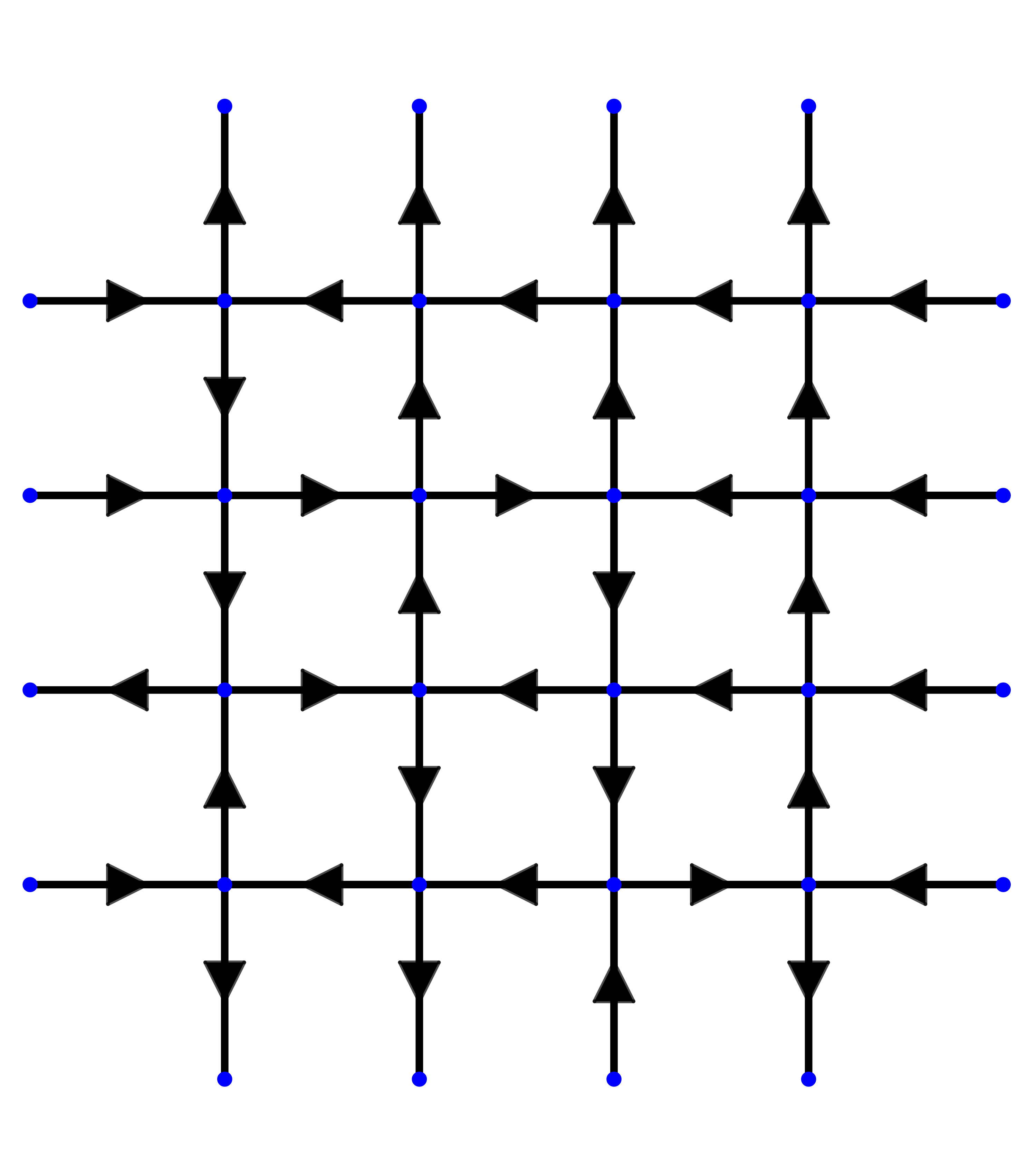}

\includegraphics[width=1.75in]{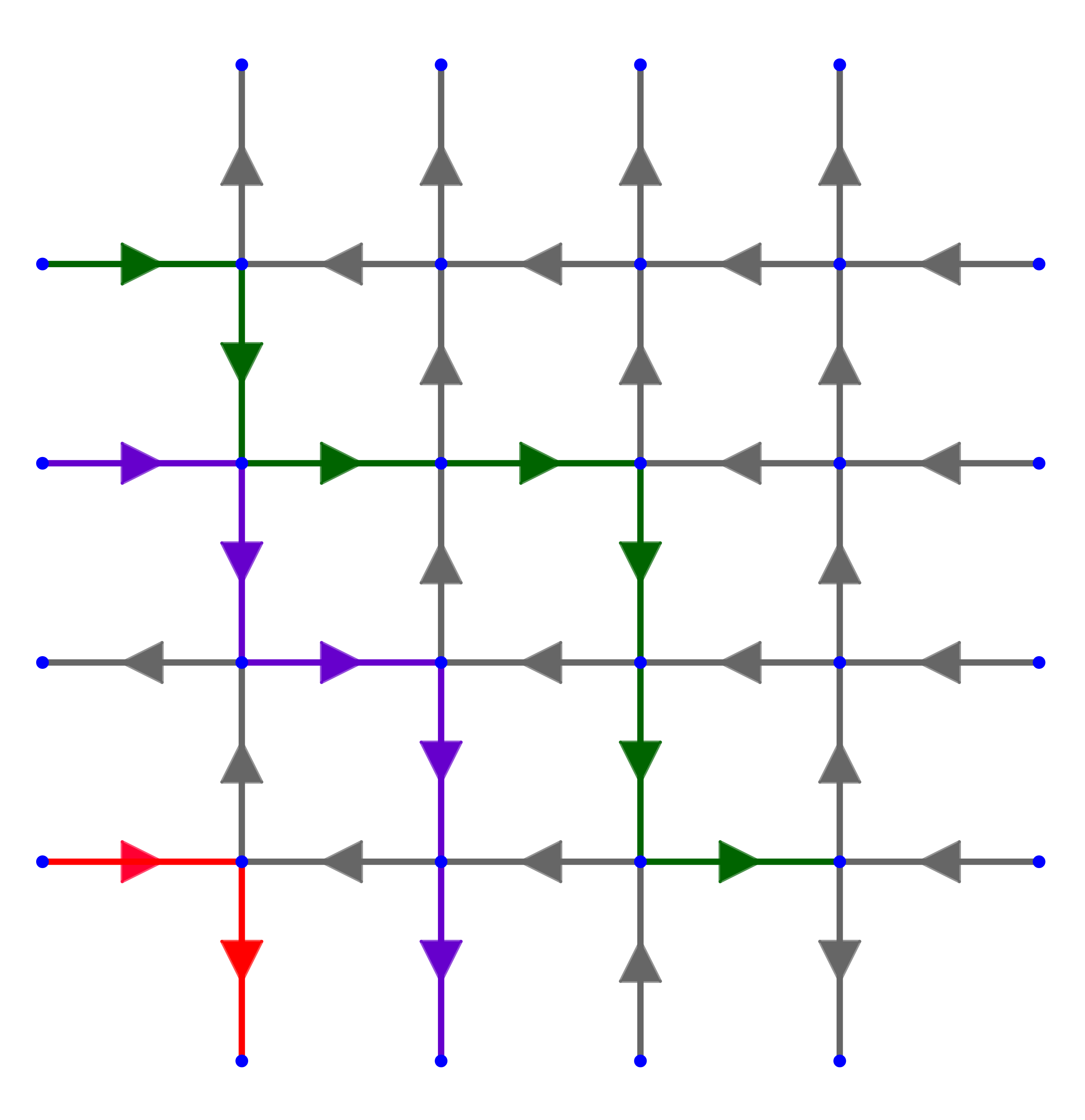} \hspace{0.5in} \includegraphics[width=1.75in]{oscupaths2.png}
\caption{An example of mapping a partial alternating sign matrix to a nest of oscullating paths. Starting with the matrix (top left), we construct a reflection of its rectangular ice configuration (top right). This can be done directly using the correspondence mentioned in Remark~\ref{rem:corresp}, or by applying the bijections in Lemmas~\ref{pasm-phf}, \ref{pfpl-phf}, and \ref{pfpl-ice} in succession, followed by a horizontal reflection. We then create the first path by starting with the topmost inward-facing directed edge along the left, and follow the directed edges to the right when possible and down otherwise. Each path is constructed similarly for each other inward-facing directed edge along the left (bottom left). Note that a new path touching the edge of a previous path forces the path to go down to satisfy the osculating condition. Finally, we delete the first and last edge of each path (bottom right).}
\label{fig:paths}
\end{figure}

\begin{figure}[hbtp]
\centering
\includegraphics[height=1.75in]{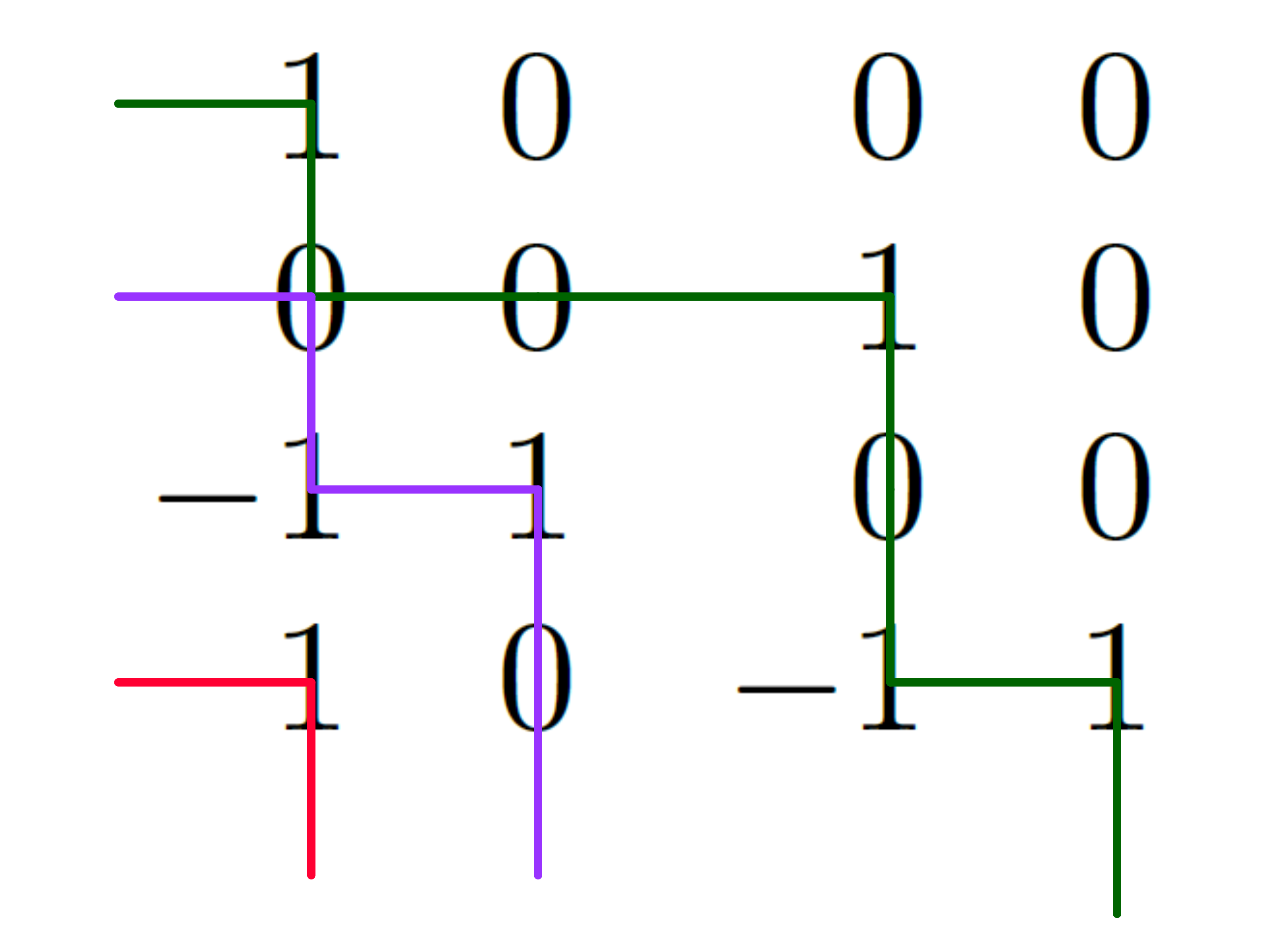} \hspace{0.5in} \includegraphics[width=1.75in]{oscupaths2.png}
\caption{An example of a more direct way to obtain a nest of oscullating lattice paths from a partial alternating sign matrix. Starting outside the matrix on the left and top, if the first nonzero entry seen to the right is a 1, create the path by moving to the right until the 1 or a previous path is reached, at which point the path moves down until the first nonzero entry to the right is a 1 (at which point the path moves right again and continues this process) or the path exits the matrix on the bottom. Each subsequent path is created in the same manner. Note that if the first nonzero entry in a row is not 1, no path will start in that row. Finally, delete the first and last edges in order to match the definition of nests of oscullating lattice paths.}
\label{fig:directpaths}
\end{figure}


This correspondence gives us the following enumeration for partial alternating sum matrices with total sum $1$.

\begin{corollary} The number of $m \times n$ partial alternating sign matrices with sum $1$ is $\binom{m+n}{m}-1$.
\end{corollary}

\begin{proof}
By Lemma~\ref{pasm-olp}, the set of $m \times n$ partial alternating sign matrices with sum $1$ is in bijection with the set of $(m,n)$-nests of osculating paths with $1$ path. That is, the set of paths in an $m \times n$ grid with starting point on the left and ending point on the bottom of the grid. Each of these paths can be uniquely extended to a lattice path in the $(m+1) \times (n+1)$ grid (where the additional column has been added on the left and additional row on the bottom) which starts in the upper left corner and ends in the lower right corner. To do so, the start of the path gets extended one step to the left, and then up to the upper left corner, and the end of the path gets extended one step down and then right to the lower right corner. This will account for all such paths in the $(m+1) \times (n+1)$ grid, except for the one which is all south steps followed by all east steps. (This path is excluded because no path starting on the left edge of the $m \times n$ grid can have its starting point as the lower left corner of the $(m+1) \times (n+1)$ grid.) Since those paths are counted by $\binom{m+n}{m}$, the total number of $m \times n$ partial alternating sign matrices with sum $1$ is $\binom{m+n}{m}-1$.
\end{proof}

\begin{openproblem}
We leave it as an open question to enumerate $\{M\in \pasm_{m,n} \ | \ \mbox{sum}(M)=t\}$ for $t>1$. Figure~\ref{nt_table} gives the number of $n\times n$ partial alternating sign matrices with total sum $t$, for $n,t\leq 6$ calculated using SageMath. Note the value for $n=6$, $t=2$ is prime, so there will not be a product formula in general, though this does not preclude the existence of a sum formula. It would be interesting to find a formula for the cardinality of the set $\{M\in \pasm_{m,n} \ | \ \mbox{sum}(M)=t\}$, since for $t=m=n$, this is the set of $n\times n$ alternating sign matrices, enumerated by $\displaystyle\prod_{j=0}^{n-1}\displaystyle\frac{(3j+1)!}{(n+j)!}$~\cite{Zeilberger,Kuperberg}.
\end{openproblem}

\begin{figure}[hbtp]
\centering
\begin{tabular}{|c|c|c|c|c|c|c|c|}
\hline
\diaghead(1,-1)%
   {\theadfont nnn}%
   {$n$}{$t$} & 0 & 1 & 2 & 3 & 4 & 5 & 6 \\ \hline
1 & 1 & 1& & & & & \\ \hline
2 & 1 & 5 & 2 & & & & \\ \hline
3 & 1 & 19 & 35 & 7 & & & \\ \hline
4 & 1 & 69 & 425 & 387 & 42 & & \\ \hline
5 & 1 & 251 & 4845 & 13861 & 7007 & 429 & \\ \hline
6 & 1 & 923 & 55897 & 458263 & 709242 & 210912 & 7436 \\ \hline
\end{tabular}
\caption{The number of $n\times n$ partial alternating sign matrices with total sum $t$.}
\label{nt_table}
\end{figure}

\begin{proof}[Proof of Theorem~\ref{thm:bijections}]
Explicit bijections between the listed objects are given in Lemmas \ref{pasm-pmt}, \ref{pasm-phf}, \ref{pfpl-phf}, \ref{pfpl-ice}, \ref{phf-oi}, and \ref{pasm-olp}.
\end{proof}

\section{Partial Alternating Sign Matrix Dynamics}
\label{sec:pasm_dyn}

In this section we explore dynamics related to partial alternating sign matrices, inspired by work of Striker \cite{Striker1}. We first describe a local move on partial fully-packed loop configurations (recall Definition \ref{pfpl}), which leads to the definition of an action called \emph{gyration} on these configurations. This is analogous to Wieland's gyration of fully-packed loops \cite{Wieland}, as discussed in Section~\ref{sec:prelim}. We then show how gyration acts on the corresponding height-function matrices and order ideals. Finally, we use Theorem~\ref{rotate} to prove a rotation-like result on \emph{partial link patterns}.

Given an $(m,n)$-partial fully-packed loop configuration, call a square \emph{even} (resp.\ \emph{odd}) if the vertex in its upper-left corner (or lower-right corner) is even (resp.\ odd). Call a square \emph{interior} if all of its surrounding vertices are interior, and \emph{exterior} otherwise. Additionally, we say a square is a \emph{boundary square} if it is in the first row or column.

Define a \emph{local action} on a square as follows. For an interior square, if its only edges are a pair of parallel lines, swap it to a pair of parallel lines in the other orientation. For an exterior square, we define the local action based on location. If the square is on the far right (and not a corner), we swap only the left edge with the pair of horizontal parallel edges (and vice versa). If the square is on the bottom (and not a corner), we swap only the top edge with the pair of vertical parallel edges (and vice versa). If the square is the bottom-right corner, we swap only the left edge with only the top edge (and vice versa). In all other cases, we do nothing. Note that with this definition, the action will never act on a boundary square. See Figure~\ref{localmoves} for an example of each of the local actions described above.

\begin{definition}
\label{fplgyration}
Define the action \emph{gyration} on an $(m,n)$-partial fully-packed loop configuration by first performing the local action on all even squares, then on all odd squares. We denote the gyration action as $G$.
\end{definition}

An example of this action is given in Figure~\ref{fpl-gyration}. Note that by construction, since this action never changes the total degree of any interior vertex and does not affect the fixed boundary, it always produces another $(m,n)$-partial fully-packed loop configuration.

\begin{figure}[hbtp]
\centering
\begin{multicols}{4}
\includegraphics[width=1.25in]{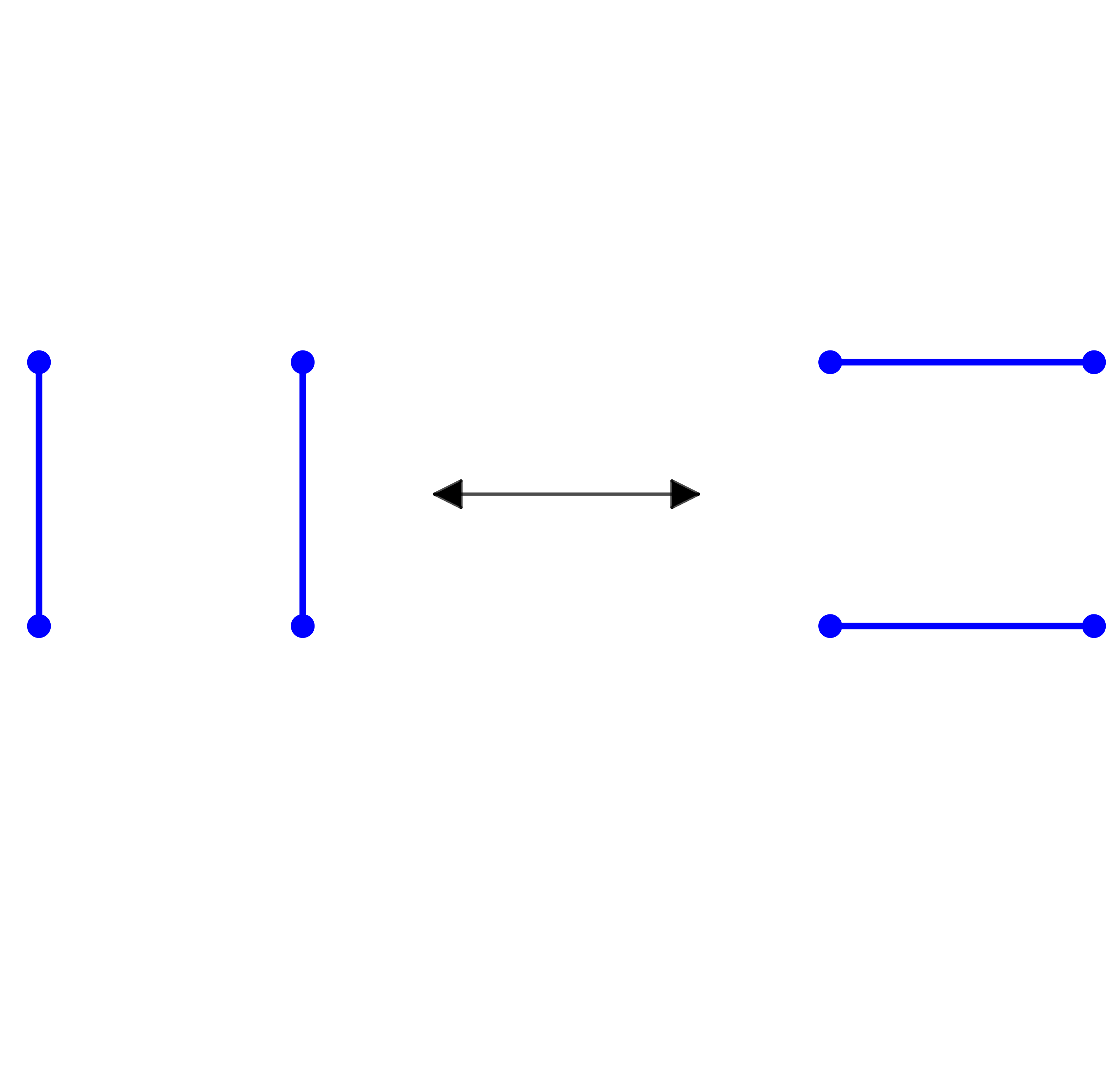}

(interior)

\includegraphics[width=1.25in]{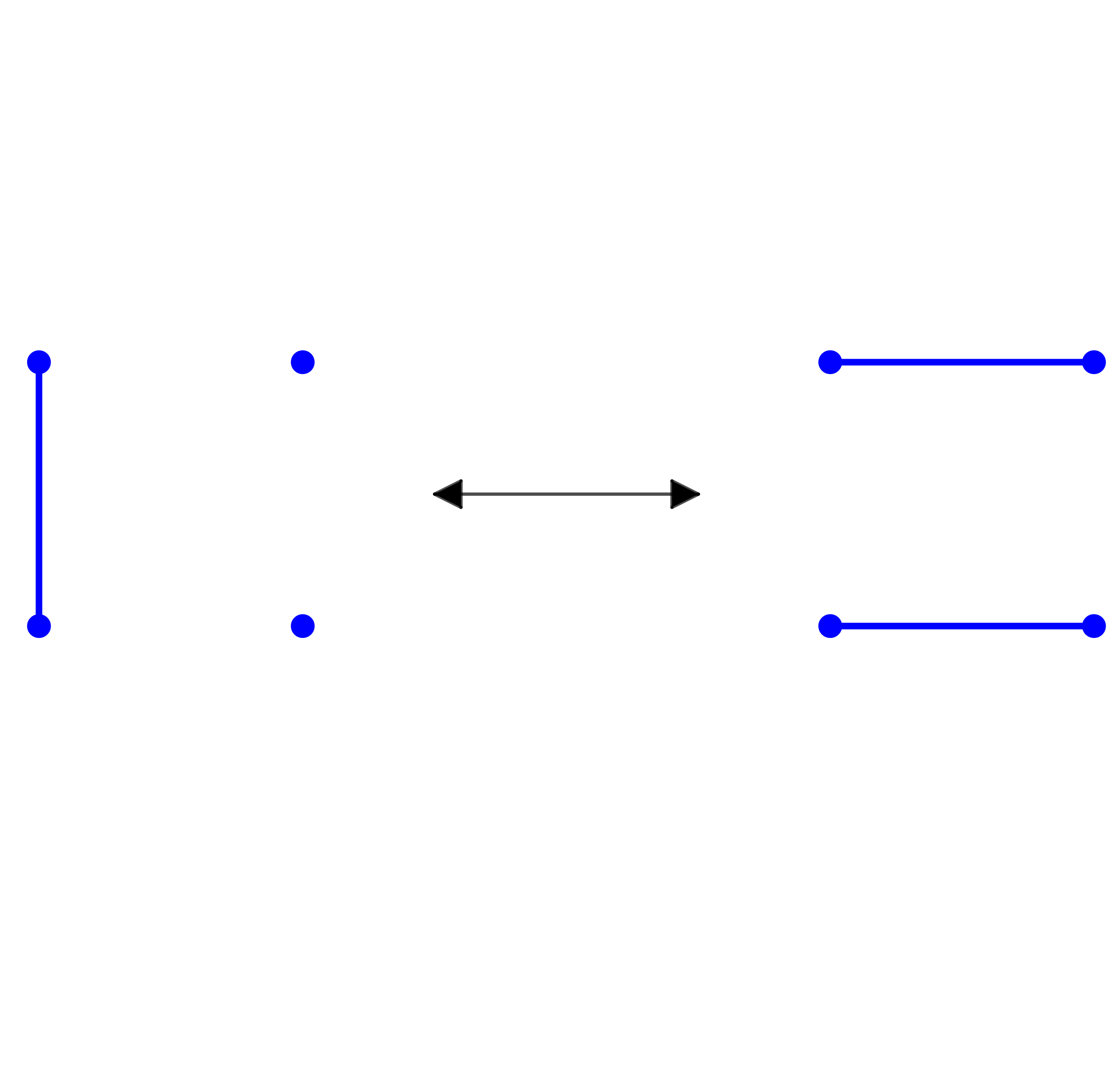}

(right-exterior)

\includegraphics[width=1.25in]{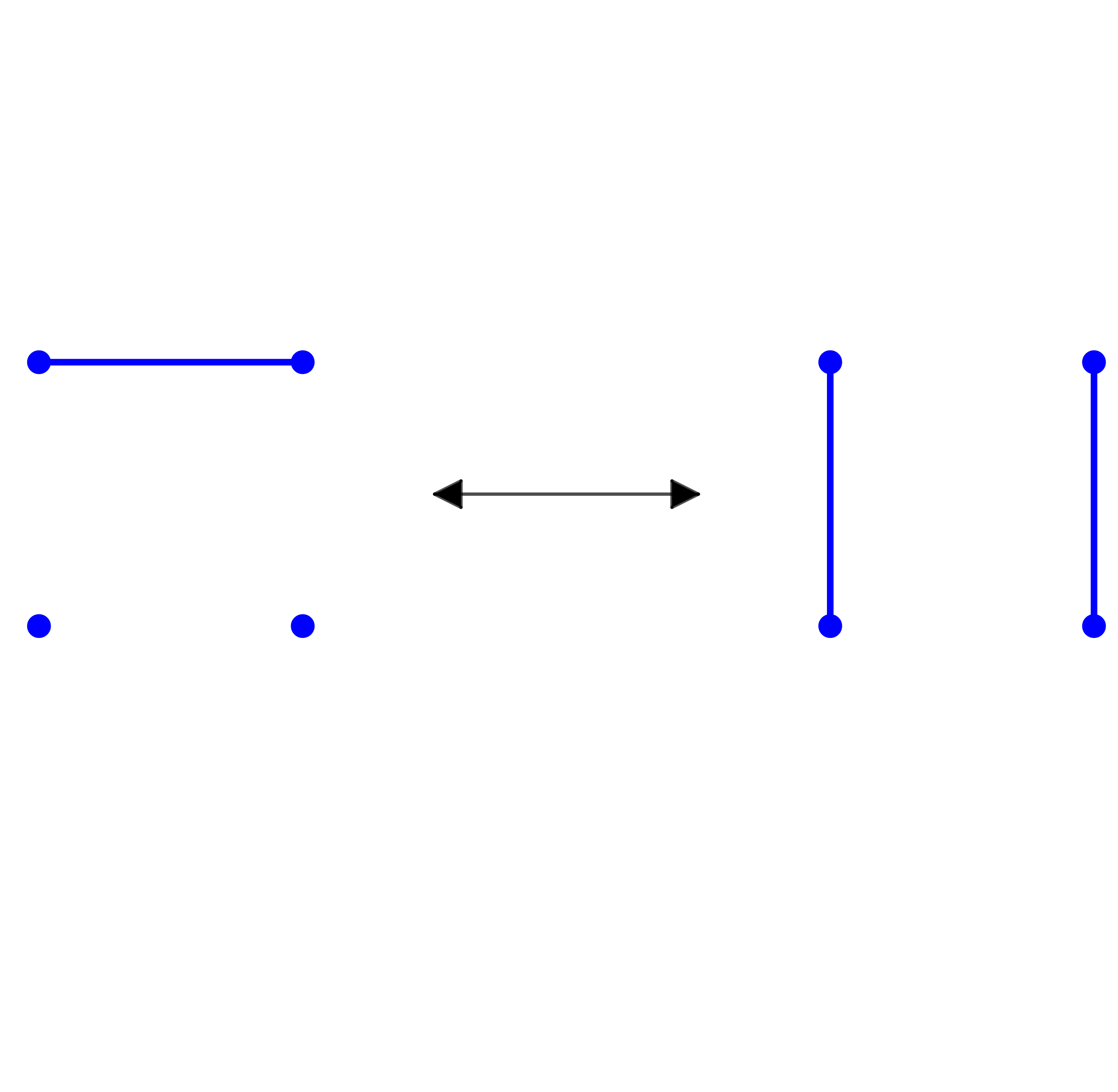}

(bottom-exterior)

\includegraphics[width=1.25in]{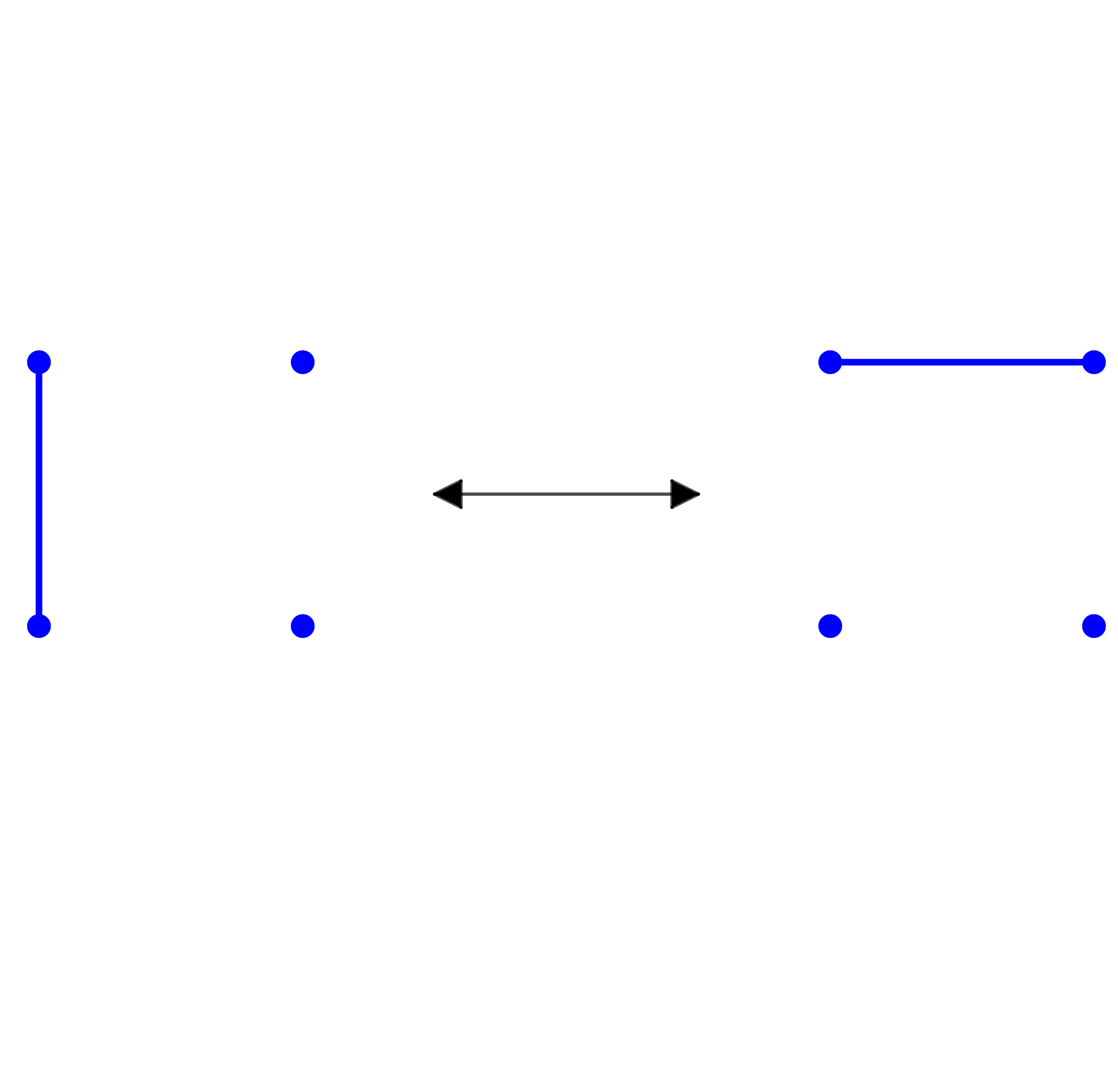}

(bottom-right-exterior)
\end{multicols}

\caption{From left to right: the local moves on interior squares, exterior squares along the right side, exterior squares along the bottom side, and the exterior square in the bottom-right corner.}
\label{localmoves}
\end{figure}

\begin{figure}[hbtp]
\centering
\includegraphics[scale=.18, valign=c]{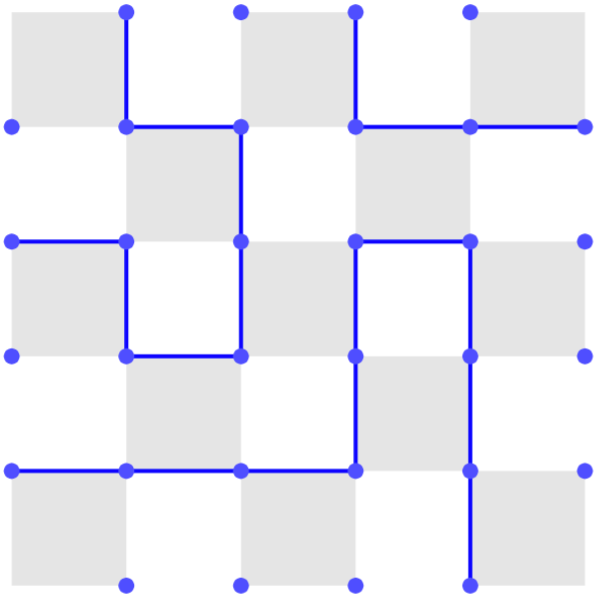} \hspace{0.1cm} $\overset{\mathfrak{g}_e}{\longrightarrow}$ \hspace{0.1cm}
\includegraphics[scale=.18, valign=c]{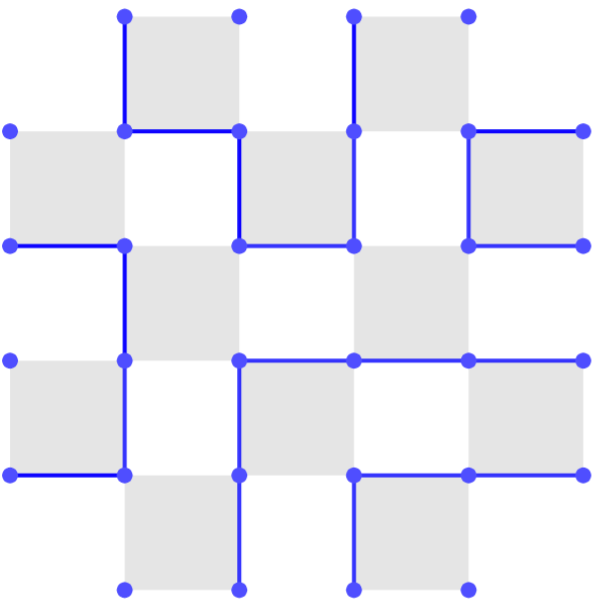} \hspace{0.1cm} $\overset{\mathfrak{g}_o}{\longrightarrow}$ \hspace{0.1cm}
\includegraphics[scale=.18, valign=c]{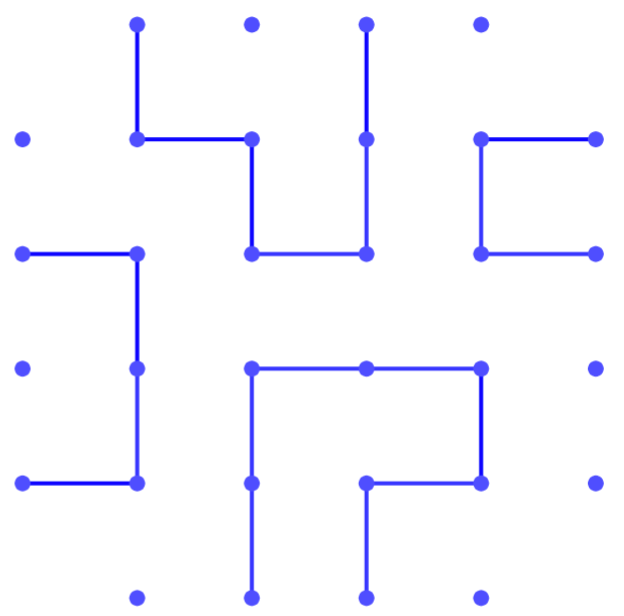}
\caption{The gyration action on a partial fully-packed loop configuration. $\mathfrak{g}_e$ denotes performing the local action on the even squares (shaded in the first diagram), and $\mathfrak{g}_o$ denotes performing the local action on the odd squares (shaded in the second diagram).}
\label{fpl-gyration}
\end{figure}

Given the results of the previous section, we can examine how the local action described above affects the different objects in bijection with partial fully-packed loop configurations. In particular, we can describe exactly what happens on the corresponding partial height-function matrices and order ideals. Recall $S_{i,j}$ from the proof of Lemma~\ref{phf-oi}.

\begin{lemma}\label{equiv}
Let $\left(h_{i,j}\right)$ be an $(m,n)$-partial height-function matrix, $\mathcal{O}\in J\left(\textbf{P}_{m,n}\right)$ the corresponding order ideal, and $F$ the corresponding $(m,n)$-partial fully-packed loop, via the bijections of Lemmas \ref{pfpl-phf} and \ref{phf-oi}. Then the following are equivalent:
\begin{enumerate}
\item[(1)] The local action applied to $F$ at the square in row $i$ and column $j$, where the rows (columns) are numbered from the top (left) starting at 1.
\item[(2)] Incrementing or decrementing the partial height-function matrix entry $h_{i,j}$ by 2, if possible.
\item[(3)] Toggling $S_{i,j}$.
\end{enumerate}
\end{lemma}

\begin{proof}
To see the equivalence between $(1)$ and $(2)$, we follow the bijection in Lemma~\ref{pfpl-phf}. We see that when the local action would change edges on $F$, this corresponds to $h_{i,j}$ being surrounded by the same number (each entry above, below, to the right, and to the left is either $h_{i,j}-1$ or $h_{i,j}+1$). Making the local action change in edges corresponds exactly to changing the value $h_{i,j}$ to the only other possible height-function value: if the surrounding values are all $h_{i,j}-1$, then $h_{i,j}$ gets changed to $h_{i,j}-2$, and if the surrounding values are all $h_{i,j}+1$, then $h_{i,j}$ gets changed to $h_{i,j}+2$. In the case where the local action would do nothing, this corresponds to $h_{i,j}$ having surrounding values which are not all the same.

To see the equivalence between $(2)$ and $(3)$, we follow the bijection in Lemma~\ref{phf-oi}. We see that incrementing $h_{i,j}$ by $2$ (and having it stay a partial height-function matrix) is only possible when removing an element from $S_{i,j}$ (in $\mathcal{O}$) results in another order ideal. Likewise, decrementing $h_{i,j}$ by $2$ is only possible when adding an element of $S_{i,j}$ results in another order ideal. When it is not possible to change $h_{i,j}$ and have the result be a partial height-function matrix, then adding or removing an element from $S_{i,j}$ results in a subset that is not an order ideal.
\end{proof}
\noindent See Figure~\ref{togglegyrex} for an example of this correspondence.

\begin{figure}[hbt]
\centering
\includegraphics[scale=0.35]{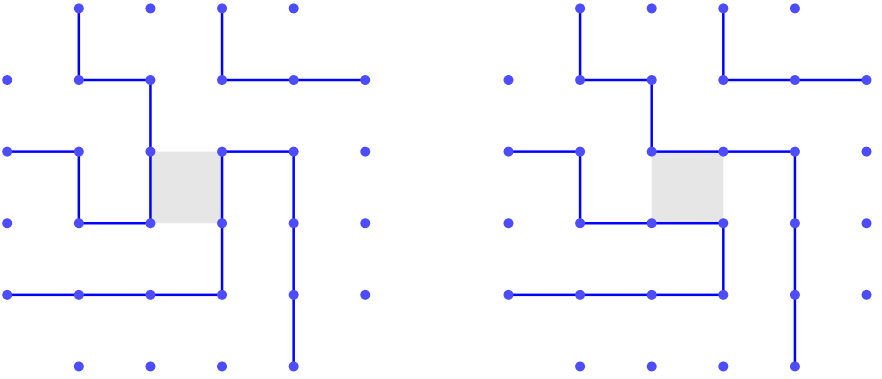}

\vspace{0.25in}

$\begin{pmatrix}
0 & 1 & 2 & 3 & 4\\
1 & 2 & 3 & 4 & 3\\
2 & 3 & \blue{2} & 3 & 2\\
3 & 4 & 3 & 2 & 3\\
4 & 3 & 4 & 3 & 2
\end{pmatrix}$
\hspace{0.5in}
$\begin{pmatrix}
0 & 1 & 2 & 3 & 4\\
1 & 2 & 3 & 4 & 3\\
2 & 3 & \blue{4} & 3 & 2\\
3 & 4 & 3 & 2 & 3\\
4 & 3 & 4 & 3 & 2
\end{pmatrix}$

\vspace{0.25in}

\includegraphics[scale=0.5]{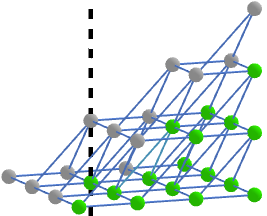} \hspace{0.45in} \includegraphics[scale=0.5]{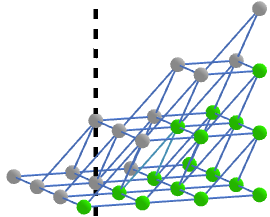}

\caption{Performing the local action on a single square of a partial fully-packed loop (top), the result on corresponding partial height-function matrices (middle), and the result on the corresponding order ideals (bottom). Note that the difference between the two order ideals is that the one on the left has the point (1,1,0) and the one on the right does not.}
\label{togglegyrex}
\end{figure}

We now define a toggle group action and show that this action corresponds to gyration on partial fully-packed loops.

\begin{definition}\label{gyr}
For any finite ranked poset $P$, define $\text{Gyr} : J(P) \rightarrow J(P)$ as the toggle group action which toggles the elements in even ranks first, then odd ranks. Its inverse, $\text{Gyr}^{-1}$, toggles the elements in odd ranks first, then even ranks.
\end{definition}

Note that this this action is well-defined because elements in ranks of the same parity do not have covering relations between them, so the corresponding toggles commute.

\begin{lemma}[Proposition 6.4, \cite{Striker1}]\label{equiv-gyr-row}
For any finite ranked poset $P$, there is an equivariant bijection (denoted by $\phi$) between $J(P)$ under \text{Row} and $J(P)$ under \text{Gyr}, that is, \text{Row} and \text{Gyr} are conjugate elements in the toggle group $T(P)$.
\end{lemma}

In the case where $P = \textbf{P}_{m,n}$, we have the following lemma, which follows from Lemma~\ref{phf-oi} and Lemma~\ref{equiv}.

\begin{lemma}\label{gyr-poset-fpl}
Let $F$ be an $(m,n)$-partial fully pack loop configuration. When $m+n$ is even, $\phi(G(F)) = \text{Gyr}(\phi(F))$. When $m+n$ is odd, $\phi(G(F)) = \text{Gyr}^{-1}(\phi(F))$.
That is, gyration on $(m,n)$-partial fully-packed loop configurations is equivalent to \text{Gyr} acting on $J\left(\textbf{P}_{m,n}\right)$ when $m+n$ is even, and $\text{Gyr}^{-1}$ acting on $J\left(\textbf{P}_{m,n}\right)$ when $m+n$ is odd.
\end{lemma}

\begin{proof}
Notice that the parity of a square in an $(m,n)$-partial fully-packed loop configuration corresponds to the parity of rank in $\textbf{P}_{m,n}$. When $m+n$ is even, even squares correspond to even ranks, and when $m+n$ is odd, even squares correspond to odd ranks. Thus, when we perform gyration on an $(m,n)$-partial fully-packed loop (which applies the local move on all of the even squares and then all of the odd squares), by Lemma~\ref{equiv}, this corresponds to toggling even ranks followed by odd ranks when $m+n$ is even, which is exactly Gyr from Definition~\ref{gyr}. When $m+n$ is odd, performing gyration on an $(m,n)$-partial fully-packed loop corresponds to toggling odd ranks followed by even ranks, which is $\text{Gyr}^{-1}$.
\end{proof}

The previous two propositions give the following theorem, which is an analogue of Theorem~8.13 in \cite{ProRow}.

\begin{theorem}
\label{thm:equivbij}
$J\left(\textbf{P}_{m,n}\right)$ under \text{Row} and $(m,n)$-partial fully-packed loops under gyration are in equivariant bijection.
\end{theorem}

Recall Theorem~\ref{rotate}, which says that in the usual $n \times n$ alternating sign matrix setting, gyration acting on a fully-packed loops rotates the link pattern. We close this section with a result on partial fully-packed loops that follows from this.

\begin{definition}
Given an $(m,n)$-partial fully-packed loop configuration, $F$, label the places where the paths exit the graph along the left and top (the fixed boundary conditions) with the numbers $\left\{1,\ldots, \lfloor\frac{m}{2}\rfloor+\lceil\frac{n}{2}\rceil\right\}$, starting with $1$ in the lower left. Arrange these numbers in a circular arc, and connect any numbers that are connected by a path in $F$ with an arc. Call this the \emph{partial link pattern} for $F$.
\end{definition}
See Figure~\ref{linkpattern} for examples of partial link patterns.

\begin{corollary}
Gyration on $(m,n)$-partial fully-packed loop configurations exhibits a partial rotation on the corresponding partial link patterns. Specifically, if $i$ and $j$ are connected in a partial link pattern, then in the corresponding partial link pattern after gyration is applied, $i-1$ and $j-1$ will either be connected to each other or not at all.
\end{corollary}

\begin{proof}
This follows from Theorem~\ref{rotate}. Any $(m,n)$-partial alternating sign matrix can be interpreted as the upper-left corner of a larger alternating sign matrix; for any rows or columns that have a sum of $0$, we can systematically add zeros and ones to the left and below until we have a larger matrix with each row and column summing to $1$. So the corresponding $(m,n)$-partial fully-packed loop is really the upper-left corner of a larger fully-packed loop. Let $F_1$ be a partial fully-packed loop and $F_2$ be the result of applying gyration on $F_1$. Let $F_1'$, be a larger fully-packed loop for which $F_1$ is the corner, and $F_2'$ be the result after gyration is applied to $F_1'$ so that $F_2$ is the corner of $F_2'$. Let $i$ and $j$ be connected in the partial link pattern for $F_1$. Then they will also be connected in the link pattern for $F_1'$. When gyration is applied on $F_1'$, the link pattern is rotated, $i-1$ and $j-1$ will be connected in the link pattern for $F_2'$. If the path connecting them stays entirely within the corner that is $F_2$, then $i-1$ and $j-1$ will be connected in the partial link pattern for $F_2$. However, the path connecting them in $F_2'$ may leave that corner, in which case they will not be connected to anything in the partial link pattern for $F_2$.
\end{proof}
\begin{figure}[hbtp]
\centering
\includegraphics[scale=0.7]{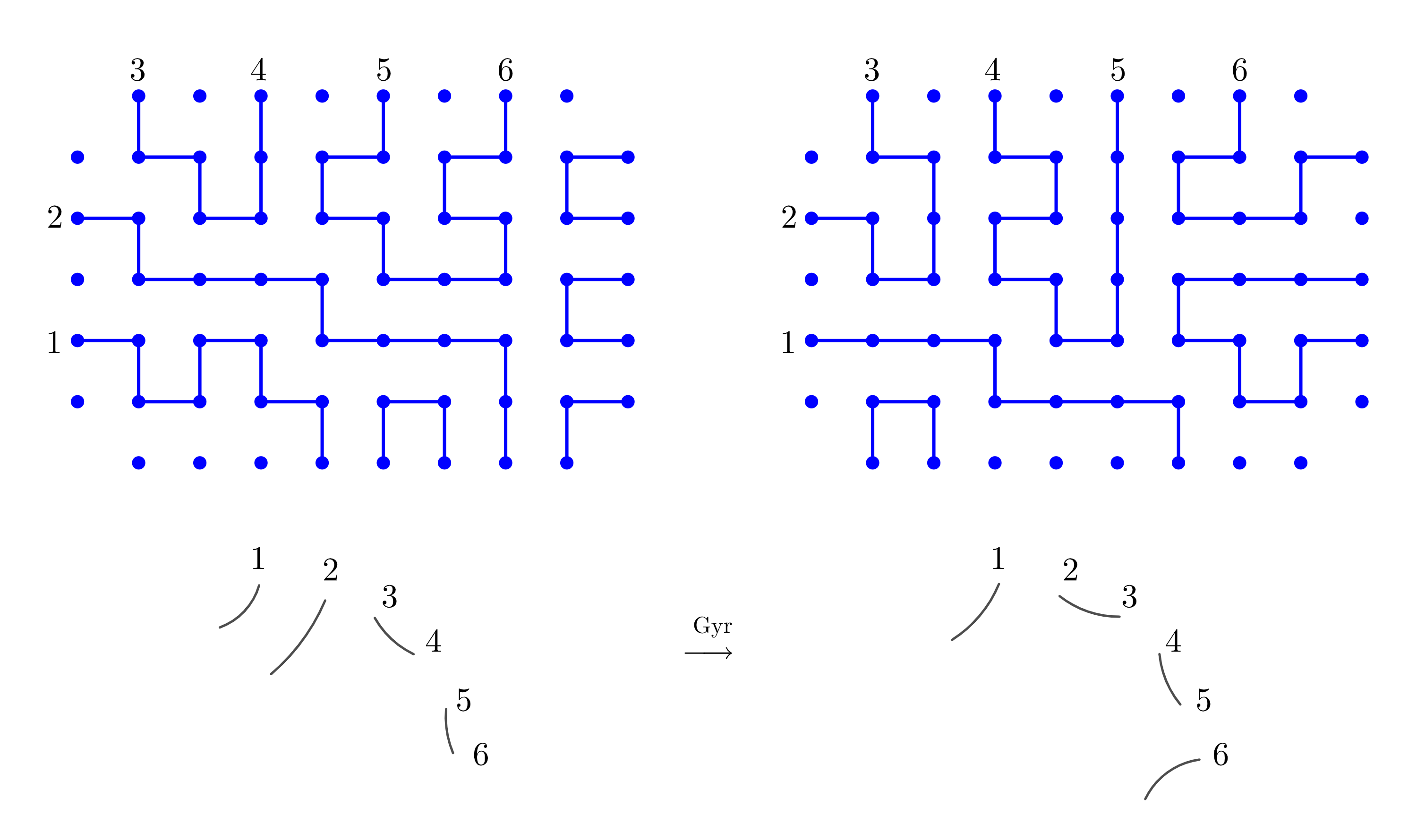}
\caption{A $(5,7)$-partial fully-packed loop configuration along with its partial link pattern (left), and the result of each after gyration is applied (right). For any numbers that do not have a connection, we choose to draw a line that terminates so that it does not cross any other lines or connect to any other numbers, to more closely mimic the feeling of being a piece of a larger link pattern. Note that these lines also ``rotate'' since they will occupy the positions not used by the connected ones.}
\label{linkpattern}
\end{figure}

\begin{remark}
Note that rowmotion orbit sizes for $\text{P}_{m,n}$ do not appear to be nice. For example, when $m=3$ and $n=5$, there are $27$ orbits of size $8$, $2$ orbits each of sizes $2$ and $9$, and $1$ orbit each of sizes $14$, $18$, $36$, $46$, $52$, and $58$, which gives the order of the map to be $4,370,184$ in this case.
\end{remark}

\section{Acknowledgments}
This work is based on research which is a part of the author's doctoral dissertation at North Dakota State University under the supervision of Dr.~Jessica Striker. The author thanks the anonymous referee for their many helpful comments which improved the quality and clarity of this paper, as well as the developers of SageMath \cite{sage}, which made many of the calculations possible.


\begin{thebibliography}{99}

\bibitem{oeis1}
OEIS Foundation~Inc. (2024).
\newblock {\em The On-Line Encyclopedia of Integer Sequences}.
\newblock \url{http://oeis.org/A202751}.

\bibitem{Behrend2}
R.~E. Behrend.
\newblock Osculating paths and oscillating tableaux.
\newblock {\em Electron. J. Combin.}, 15(1):Research Paper 7, 60, 2008.

\bibitem{Bressoud}
D.~M. Bressoud.
\newblock {\em Proofs and confirmations: The story of the alternating sign
  matrix conjecture}.
\newblock MAA Spectrum. Mathematical Association of America, Washington, DC;
  Cambridge University Press, Cambridge, 1999.

\bibitem{brouwer-schrijver}
A. Brouwer and A. Schrijver.
\newblock On the period of an operator, defined on antichains.
\newblock {\em Math Centrum report ZW}.

\bibitem{Brualdi-Dahl}
R.~A. Brualdi and G. Dahl.
\newblock Sign-restricted matrices of 0's, 1's, and -1's.
\newblock {\em Linear Algebra Appl.}, 615:77--103, 2021.

\bibitem{Cameron-Flaass}
P.~J. Cameron and D.~G. Fon-Der-Flaass.
\newblock Orbits of antichains revisited.
\newblock {\em European J. Combin.}, 16(6):545--554, 1995.

\bibitem{DilksPechenikStriker}
K.~Dilks, O.~Pechenik, and J.~Striker.
\newblock Resonance in orbits of plane partitions and increasing tableaux.
\newblock {\em J. Combin. Theory Ser. A}, 148:244--274, 2017.

\bibitem{DilksStrikerVorland}
K. Dilks, J. Striker, and C. Vorland.
\newblock Rowmotion and increasing labeling promotion.
\newblock {\em J. Combin. Theory Ser. A}, 164:72--108, 2019.

\bibitem{Duchet}
P. Duchet.
\newblock Sur les hypergraphes invariants.
\newblock {\em Discrete Math.}, 8(3), 1974.

\bibitem{fonder1}
D.~G. Fon-Der-Flaass.
\newblock Orbits of antichains in ranked posets.
\newblock {\em European J. Combin.}, 14(1), 1993.

\bibitem{Fortin}
M. Fortin.
\newblock The {M}ac{N}eille completion of the poset of partial injective
  functions.
\newblock {\em Electron. J. Combin.}, 15(1):Research paper 62, 30, 2008.

\bibitem{Chained}
D. Heuer, C. Morrow, B. Noteboom, S. Solhjem, J. Striker, and
  C. Vorland.
\newblock Chained permutations and alternating sign matrices---inspired by
  three-person chess.
\newblock {\em Discrete Math.}, 340(12):2732--2752, 2017.

\bibitem{Kuperberg}
G. Kuperberg.
\newblock Another proof of the alternating-sign matrix conjecture.
\newblock {\em Internat. Math. Res. Notices}, (3):139--150, 1996.

\bibitem{Lascoux}
A. Lascoux and M. Sch\"{u}tzenberger.
\newblock Treillis et bases des groupes de {C}oxeter.
\newblock {\em Electron. J. Combin.}, 3(2):Research paper 27, approx. 35, 1996.

\bibitem{MRRASM}
W.~H. Mills, D.~P. Robbins, and H. Rumsey, Jr.
\newblock Alternating sign matrices and descending plane partitions.
\newblock {\em J. Combin. Theory Ser. A}, 34(3):340--359, 1983.

\bibitem{Propp}
J. Propp.
\newblock The many faces of alternating-sign matrices.
\newblock In {\em Discrete models: combinatorics, computation, and geometry
  ({P}aris, 2001)}, Discrete Math. Theor. Comput. Sci. Proc., AA, pages
  043--058. Maison Inform. Math. Discr\`et. (MIMD), Paris, 2001.

\bibitem{ProppRoby}
J. Propp and T. Roby.
\newblock Homomesy in products of two chains.
\newblock {\em Electron. J. Combin.}, 22(3):Paper 3.4, 29, 2015.

\bibitem{CSP}
V. Reiner, D. Stanton, and D. White.
\newblock The cyclic sieving phenomenon.
\newblock {\em J. Combin. Theory Ser. A}, 108(1):17--50, 2004.

\bibitem{Roby}
T. Roby.
\newblock Dynamical algebraic combinatorics and the homomesy phenomenon.
\newblock In {\em Recent trends in combinatorics}, volume 159 of {\em IMA Vol.
  Math. Appl.}, pages 619--652. Springer, [Cham], 2016.

\bibitem{stanley1}
R.~P. Stanley.
\newblock {\em Enumerative combinatorics. {V}olume 1}.
\newblock Cambridge Studies in Advanced Mathematics. Cambridge University
  Press, Cambridge, {S}econd edition, 2012.

\bibitem{sage}
W.~A. Stein et~al.
\newblock {\em {S}age {M}athematics {S}oftware ({V}ersion 10.2)}.
\newblock The Sage Development Team, 2023.
\newblock \url{http://www.sagemath.org}.

\bibitem{Striker1}
J. Striker.
\newblock The toggle group, homomesy, and the {R}azumov--{S}troganov
  correspondence.
\newblock {\em Electron. J. Combin.}, 22(2):Paper 2.57, 17, 2015.

\bibitem{StrikerNotices}
J. Striker.
\newblock Dynamical algebraic combinatorics: promotion, rowmotion, and
  resonance.
\newblock {\em Notices Amer. Math. Soc.}, 64(6):543--549, 2017.

\bibitem{ProRow}
J. Striker and N. Williams.
\newblock Promotion and rowmotion.
\newblock {\em European J. Combin.}, 33(8):1919--1942, 2012.

\bibitem{Wieland}
B. Wieland.
\newblock A large dihedral symmetry of the set of alternating sign matrices.
\newblock {\em Electron. J. Combin.}, 7:Research Paper 37, 13, 2000.

\bibitem{Zeilberger}
D. Zeilberger.
\newblock Proof of the alternating sign matrix conjecture.
\newblock {\em Electron. J. Combin.}, 3(2):Research Paper 13, 1996.
\newblock The Foata Festschrift.

\end{thebibliography}
\end{document}